\documentclass[twoside]{article}
\usepackage[utf8]{inputenc}

\usepackage{graphicx}
\usepackage{amsmath,amsthm,amssymb,mathtools}
\usepackage{xcolor}
\usepackage{booktabs}
\usepackage{array}
\usepackage{longtable}
\usepackage{titlesec}
\usepackage[breaklinks=true]{hyperref}
\usepackage{makecell}

\usepackage{amsmath,amsthm,amssymb,mathtools}
\usepackage{xcolor}
\usepackage{booktabs}
\usepackage[breaklinks=true]{hyperref}
\usepackage{array}
\usepackage{subcaption}
\usepackage{longtable}

\newcommand{\R}{\mathbb{R}}
\newcommand{\trace}{\operatorname{tr}}

\newtheorem{thm}{Theorem}
\newtheorem{lemma}[thm]{Lemma}

\newtheorem{corollary}[thm]{Corollary}

\newcommand{\ns}[1]{\makebox[2em][r]{\,(#1)}}   

\newcommand{\exclude}[1]{}

\newcommand{\matip}[2]{#1 \bullet #2}

\newcommand{\citep}[1]{\cite{#1}}
\newcommand{\citet}[1]{\cite{#1}}
\usepackage{geometry}
\usepackage{fancyhdr}
\makeatletter
\g@addto@macro\normalsize{%
  \setlength{\abovedisplayskip}{8pt plus 1pt minus 1pt}
  \setlength{\belowdisplayskip}{8pt plus 1pt minus 1pt}
  \setlength{\abovedisplayshortskip}{8pt plus 1pt minus 1pt}
  \setlength{\belowdisplayshortskip}{8pt plus 1pt minus 1pt}
}
\makeatother

\title{Sparse Cuts for the Positive Semidefinite Cone}
\author{Oktay G\"unl\"uk$^a$ \and Paul J\"unger$^b$ \and {Jeff Linderoth}$^c$ \and Andrea Lodi$^b$ \and James Luedtke$^c$}
\date{$^a$Georgia Tech, $^b$Cornell Tech, $^c$U. Wisconsin-Madison \\[.3cm] \today}

\begin{document}
	
	\maketitle
	\begin{abstract}
		We consider optimization problems containing nonconvex quadratic functions for which semidefinite programming (SDP) relaxations often yield strong bounds.
		We investigate linear inequalities that outer approximate the positive semidefinite cone and are sparse in the sense that they are supported only on the variables corresponding to products of variables present in quadratic functions. We show that these sparse linear inequalities yield an LP relaxation that gives the same bound as the SDP relaxation.
		We demonstrate how to identify these inequalities via a separation procedure that involves solving a structured ``projection'' SDP.  In a computational study, we find that the sparse LP relaxations defined by these inequalities can accelerate branch-and-bound methods for globally solving nonconvex optimization problems.
	\end{abstract}

	\newcommand{\BnBx}{B\&B}
	\newcommand{\BnB}{B\&B }
	\newcommand{\PSDx}{PSD}
	\newcommand{\PSD}{PSD }
	\newcommand{\SDPx}{SDP}
	\newcommand{\SDP}{SDP }
	\newcommand{\LPx}{LP}
	\newcommand{\LP}{LP }
	
	\section{Introduction}
	\label{sec:intro}

	We propose a computational mechanism to accelerate the global solution of sparse nonconvex optimization problems that contain nonconvex quadratic functions. For clarity, we focus on a quadratically constrained quadratic program (QCQP) of the form
	\begin{subequations}
		\label{eq:nonconvex-qp}
		\begin{align}
			z^{\mathrm{QP}} := \min_{x\in\mathcal X} \quad & x^\top \bar Q^0 x + c_0^\top x + d_0 \label{eq:qp-obj}\\
			\text{s.t.} \quad & x^\top \bar Q^k x + c_k^\top x + d_k \leq 0,
			\quad k=1,\ldots,m, \label{eq:qp-cons}
		\end{align}
	\end{subequations}
	where, for each $k\in\{0,\ldots,m\}$, $c_k\in\R^n$, $d_k\in\R$, and $\bar Q^k\in\mathcal S_n$, with $\mathcal S_n$ denoting the space of $n\times n$ real symmetric matrices. The set $\mathcal X$ is meant to encode simple linear or bound constraints on $x$. The matrices $\bar Q^k$ are not assumed to be positive semidefinite (PSD), so both the objective function and the feasible region of~\eqref{eq:nonconvex-qp} may be nonconvex. Quadratic models arise in a broad range of applications, including planning and scheduling, engineering design, power systems, control, and signal processing\citep{floudas.visweswaran:95, burer.letchford:12, pezeshki.scharf.chong:10}.
	
	\paragraph{The Shor relaxation.}
	A standard mechanism for obtaining a convex relaxation of~\eqref{eq:nonconvex-qp}, due to \citet{shor:90}, begins by introducing a symmetric matrix of decision variables $X\in\mathcal S_n$ to represent the products $xx^\top$. This gives the exact lifted reformulation
	\begin{subequations}
		\label{eq:shor-reformulation}
		\begin{align}
			z^{\mathrm{QP}} = \min_{x,X} \quad &\matip{\bar Q^0}{X} + c_0^\top x + d_0 \label{eq:shor-reformulation-obj}\\
			\text{s.t.} \quad & \matip{\bar Q^k}{X} + c_k^\top x + d_k \leq 0, \quad k=1,\ldots,m, \label{eq:shor-reformulation-cons}\\
			& x \in \mathcal X, \label{eq:shor-reformulation-X}\\
			& X = xx^\top, \label{eq:shor-reformulation-Xxx}
		\end{align}
	\end{subequations}
	where $A\bullet B:=\trace(B^\top A)$ denotes the trace inner product.
	In the Shor reformulation, the objective and quadratic constraints are linear in the lifted variables $(x,X)$, and all nonconvexity associated with the products has been isolated in~\eqref{eq:shor-reformulation-Xxx}. Define
	\begin{equation}
		\label{eq:Y-definition}
		Y(x,X):=
		\begin{bmatrix}
			1 & x^\top\\
			x & X
		\end{bmatrix}\in\mathcal S_{n+1}.
	\end{equation}
	The equality $X = xx^\top$ in~\eqref{eq:shor-reformulation-Xxx} is equivalent to requiring that $Y$ is positive semi-definite (PSD) and rank 1:
	\begin{subequations}
		\label{eq:rank-one-characterization}
		\begin{align}
			Y(x,X) &\succeq 0 \quad \text{and}\\
			\label{eq:nonconvex-rank-one}
			\operatorname{rank}(Y(x,X)) &= 1.
		\end{align}
	\end{subequations}
	We use the standard notation $A \succeq 0$ to denote that the symmetric matrix $A$ is PSD.
	
	The Shor relaxation is obtained by dropping the nonconvex rank constraint~\eqref{eq:nonconvex-rank-one} from the formulation~\eqref{eq:shor-reformulation}, yielding
	\begin{subequations}
		\label{eq:shor-relaxation-xX}
		\begin{align}
			z^{\mathrm{SDP}} := \min_{x,X} \quad & \matip{\bar Q^0}{X}+c_0^\top x + d_0\\
			\text{s.t.}\quad & \matip{\bar Q^k}{X} + c_k^\top x + d_k \leq 0, \quad k=1,\ldots,m,\\
			& x \in \mathcal X,\\
			& \begin{bmatrix} 1 & x^\top\\ x & X \end{bmatrix} \succeq 0.
		\end{align}
	\end{subequations}
	For notational convenience, define
	\begin{equation}
		\label{eq:Q-definition}
		Q^k:= \begin{bmatrix}
			d_k & \tfrac12 c_k^\top\\
			\tfrac12 c_k & \bar Q^k
		\end{bmatrix}, \qquad k=0,\ldots,m,
	\end{equation}
	so that $\matip{Q^k}{Y(x,X)} = \matip{\bar Q^k}{X}+c_k^\top x + d_k$. 
	Let $\mathcal Y$ denote the set of symmetric matrices $Y$ satisfying $Y_{00}=1$ and whose first column (excluding $Y_{00}$) defines a vector $x\in\mathcal X$. The Shor SDP relaxation can then be written compactly as
	\begin{subequations}
		\label{eq:sdp}
		\begin{align}
			z^{\mathrm{SDP}} = \min_{Y \in \mathcal Y} \quad & \matip{Q^0}{Y} \label{eq:sdp-obj}\\
			\text{s.t.} \quad & \matip{Q^k}{Y} \leq 0, \quad k=1,\ldots,m, \label{eq:sdp-cons}\\
			& Y \in \mathcal S^+_{n+1}, \label{eq:sdp-psd}
		\end{align}
	\end{subequations}
	where $\mathcal S^+_{n+1}$ is the cone of $(n+1)\times(n+1)$ symmetric PSD matrices.  In the remainder of the paper, we omit the dimension from $\mathcal S$ and $\mathcal S^+$ when it is clear from context.
	
	The SDP relaxation~\eqref{eq:sdp} is known empirically to provide a strong lower bound on $z^{\mathrm{QP}}$~\citep{anstreicher:09}. If the variables $x$ are bounded, it can be strengthened with additional linear inequalities, including McCormick inequalities~\citep{mccormick:76} and inequalities arising from the Boolean Quadric Polytope~\citep{burer.letchford:09}. We treat such inequalities as additional linear constraints in~\eqref{eq:sdp-cons} or in the definition of $\mathcal Y$.
	
	Using the strong lower bound $z^{\mathrm{SDP}}$ to solve the nonconvex QCQP~\eqref{eq:nonconvex-qp} to global optimality comes with computational challenges. Modern solvers such as \textsc{Baron}, \textsc{Cplex}, \textsc{Gurobi}, and \textsc{Scip} typically solve nonconvex quadratic programs using spatial branch-and-bound (\BnBx) algorithms that solve a relaxation at each node of the \BnB tree~\citep{tawarmalani.sahinidis:05,bonami.gunluk.linderoth:18,bestuzheva.et.al:25}.  Although SDPs are polynomially solvable in theory, large-scale instances remain challenging in practice; see~\citet{majumdar.et.al:20} for a survey. Moreover, interior-point methods~\citep{nesterov.nemirovski:94} and first-order or regularization methods~\citep{malick.et.al:09,sun.et.al:19} can be difficult to warm-start effectively, limiting their utility when child-node relaxations differ only slightly from their parents.
	
	Our goal is therefore to obtain bounds as strong as $z^{\mathrm{SDP}}$ using a sparse linear relaxation that can be reoptimized efficiently within a \BnB solver. 
	\begin{equation}
		\label{eq:semi-infinite-psd}
		Y \in \mathcal S^+ \quad \Longleftrightarrow \quad \matip{C}{Y} \geq 0 \quad \forall C \in \mathcal S^+.
	\end{equation}
	For a finite set $\mathcal T^+\subset\mathcal S^+$, retaining only the associated inequalities gives the LP relaxation
	\begin{subequations}
		\label{eq:lp}
		\begin{align}
			z^{\mathrm{LP}} = \min_{Y \in \mathcal Y} \quad & \matip{Q^0}{Y}\\
			\text{s.t.}\quad & \matip{Q^k}{Y} \leq 0, \quad k=1,\ldots,m,\\
			& \matip{C}{Y} \geq 0, \quad C \in \mathcal T^+.
		\end{align}
	\end{subequations}
	If $\widehat{Y}$ is a solution to~\eqref{eq:lp}, and $\widehat{Y} \in \mathcal{S}^+$, then the SDP~\eqref{eq:sdp} is solved; otherwise $v^\top \widehat{Y} v < 0$ for an eigenvector $v$ of $\widehat{Y}$ associated with a negative eigenvalue.  The matrix $C := vv^\top \in \mathcal{S}^+$, so the linear inequality $C \bullet Y \geq 0$ can be added to the set $\mathcal{T}^+$, and the process can be repeated to solve the SDP relaxation via what is known as Kelley's cutting-plane algorithm \cite{kelley:60}.  Under mild conditions, there exists a finite set $\mathcal{T}^+$ for which the LP relaxation will provide the optimal relaxation bound $z^{\mathrm{SDP}}$ \cite{krishnan.mitchell:06,deroux.carr.ravi:25}.
	
	
	\paragraph{Sparse PSD cuts.}
	Several authors have investigated computationally tractable, sparse, linear or conic approximations of PSD constraints. \citet{sherali.fraticelli:02} derive semidefinite cuts that strengthen standard lifted linear relaxations.  \citet{qualizza.belotti.margot:12} study linear outer approximations for QCQPs coming from eigenvector-based inequalities, including sparsification procedures for these PSD cuts and cuts derived from principal submatrices.  \citet{hijazi.coffrin.vanhentenryck:16} develop polynomial SDP cuts based on small principal minors for optimal power flow.  \citet{baltean-lugojan.et.al:19} consider sparse PSD cuts and use a trained machine learning model to score the expected objective improvement from adding new cuts.  \citet{dey.et.al:22} formulate the search for an effective sparse eigenvector cut as a sparse principal-component-analysis problem.  Complementing this computational work, \citet{blekherman.et.al:22} analyze the approximation of the PSD cone obtained by requiring all principal submatrices of a fixed order to be PSD.
	
	These approaches demonstrate that sparsity can substantially improve the computational behavior of PSD-based cutting planes.  They make different tradeoffs among the support of a cut, the cost of generating the cut, and the strength of the resulting relaxation. However, none of these approaches can simultaneously ensure that every generated inequality is supported only on the product variables present in the original QCQP and that the resulting relaxation can recover the true SDP bound $z^{\mathrm{SDP}}$.  Our approach achieves both properties.
	
	\paragraph{Aggregate sparsity and PSD matrix completion.}
	To describe the problem's sparsity pattern, let
	\begin{equation}
		\label{eq:E-definition}
		\begin{aligned}
			E := \bigl\{(i,j) \in \{0,1,\ldots,n\}^2:\;&
			i=0\ \text{or } j=0\ \text{or } i=j \\
			&\text{or } Q^k_{ij} \neq 0
			\text{ for some } k \in \{0,\ldots,m\}\bigr\}.
		\end{aligned}
	\end{equation}
	The set of pairs $E$ contains all elements in the first row and column, on the diagonal, and for every non-zero matrix element from the quadratic forms appearing in the instance.  We assume throughout that any additional linear constraints in the definition of $\mathcal Y$ are also supported on $E$.  With the set $E$, we can associate its (undirected) \emph{sparsity graph} with vertex set $\{0,1,\ldots,n\}$ and a single edge $\{i,j\}$ for each off-diagonal pair $(i,j) \in E$.
	The entries $Y_{ij}$ with $(i,j) \in E$ form a partial symmetric matrix. Let
	\begin{equation}
		\label{eq:se-subspace-def}
		\mathcal S(E):=\{Y \in \mathcal S : Y_{ij} = 0 \text{ for all } (i,j) \notin E\},
	\end{equation}
	be the subspace of symmetric matrices with zero entries outside of $E$.  We refer to matrices in $\mathcal S(E)$ as $E$-partial matrices.
	Next, define $\Pi_E$ as the orthogonal projection from $\mathcal{S}$ onto $\mathcal{S}(E)$, i.e., for $Y \in \mathcal{S}$, 
	\[ [\Pi_E(Y)]_{ij} = \begin{cases} Y_{ij} & (i,j) \in E \\
		0 & \text{otherwise}.\end{cases} \]
	In addition, let
	\begin{equation}
		\label{eq:psd-completable-matrices}
		\Pi_E(\mathcal{S}^+) = \{ \Pi_E(Y) : Y \in \mathcal{S}^+ \}
	\end{equation}
	be the cone of PSD completable $E$-partial matrices. That is, $Z \in \Pi_E(\mathcal{S}^+)$ means that $Z$ admits a PSD completion.
	
	These definitions connect our formulation directly to the classical theory of positive semidefinite matrix completion. Foundational results of \citet{grone.et.al:84} characterize matrix completion in terms of properties of the sparsity graph, and \citet{agler.et.al:88} study the dual cone of PSD matrices with a prescribed sparsity pattern.
	Recall that a graph is called \emph{chordal} if every cycle of length four or more in the graph has a chord.
	The foundational result from~\citet{grone.et.al:84} is that if the sparsity graph of the matrix is chordal, then a partial matrix with specified diagonal admits a PSD completion exactly when every fully specified submatrix associated with a maximal clique is PSD. For a nonchordal sparsity graph, one may first construct its chordal extension, but the added edges correspond to fill entries that must be represented by additional variables.
	
	These matrix completion results have been leveraged to solve sparse SDPs more efficiently. Fukuda et al.~\cite{fukudaetal2001} develop a general matrix-completion framework that exploits the aggregate sparsity of SDP data in primal-dual interior-point methods, and Nakata et al.~\cite{nakata.et.al:03} give an implementation and numerical study. Burer~\cite{burer2003} develops an interior-point method directly in the space of partial positive semidefinite matrices. The reader is referred to Vandenberghe and Andersen~\cite{vandenberghe.andersen:15} for a comprehensive treatment of chordal graphs, PSD-completable matrices, and sparse PSD cones. These works reformulate and solve the SDP more efficiently, typically by introducing clique variables associated with a chordal extension. Our objective is different but complementary---after using SDP computations to identify strong inequalities, we seek a final \emph{linear} formulation that contains only the entries indexed by the original pattern $E$ and can therefore be reoptimized efficiently throughout a \BnB tree.  In particular, we do not want fill entries introduced by a chordal extension to appear as variables in the final formulation.
	
	The underlying theory that allows us to achieve our objective is conic duality.  In Section~\ref{sec:sparse-rel} we leverage the well-known fact that the dual of the PSD-completable cone $\Pi_E(\mathcal S^+)$, in the space $\mathcal S(E)$, is
	\begin{equation}
		\label{eq:completion-duality}
		\Pi_E(\mathcal S^+)^* = \mathcal S^+ \cap \mathcal S(E).
	\end{equation}
	Consequently, PSD completability of an $E$-partial matrix can be expressed using linear inequalities $\matip{C}{Y} \geq 0$ whose coefficients are defined by sparse PSD matrices $C\in\mathcal S^+\cap\mathcal S(E)$.  We use this conic duality to construct and separate linear inequalities whose support agrees exactly with the original QCQP sparsity pattern.  
	
	\paragraph{Contributions.}
	Building on this established matrix-completion framework, we make the following contributions. First, we leverage~\eqref{eq:completion-duality} to give a semi-infinite linear reformulation of~\eqref{eq:sdp} that has nonzeros only in $E$, and we show that the reformulation has the same optimal value as the full SDP relaxation.  Under a standard strong-duality assumption, we also show that the optimal value can be attained using finitely many PSD cuts that have support in $E$. Second, we extend this construction to doubly nonnegative relaxations when the original variables are nonnegative. Third, we formulate structured projection SDPs that separate violated inequalities in these two sparse cut families.
	
	The direct Kelley cutting-plane procedure associated with the proposed separation problems can converge slowly.  We therefore propose to first solve the SDP relaxation~\eqref{eq:sdp} and generate inequalities by separating points close to an optimal SDP solution.  Repeating this process yields a modest size collection of sparse inequalities that is added to a formulation subsequently solved by a standard \BnB solver.  The initial SDP and cut-generation costs are incurred once, with the aim of producing a sparse LP relaxation that can be reoptimized effectively at many nodes. This strategy is related in spirit to the QCR method and its variants \citep{billionnet.elloumi:07,billionnet.elloumi.lambert:12}, which also solve an auxiliary SDP to construct a formulation for a general-purpose solver. The resulting QCR formulation is a possibly dense convex quadratic program, whereas our construction produces an LP relaxation supported on $E$.
	
	Our computational results show that the SDP bound can be approached, and in many cases attained to numerical tolerance, using relatively few sparse inequalities. The final LP relaxations are substantially cheaper to solve than formulations containing large collections of dense PSD cuts. When used to augment a spatial \BnB solver, the sparse cuts improve performance substantially on instances derived from the BoxQP test set with the addition of nonconvex quadratic constraints and on several groups of QPLIB instances.
	
	The remainder of the paper is organized as follows. Section~\ref{sec:sparse-rel} develops the sparse SDP and doubly nonnegative reformulations and the finite-cut result. Section~\ref{sec:separation} presents the separation problems and the accelerated cut-generation procedure. Section~\ref{sec:computation} reports the computational experiments, including comparisons with dense PSD cuts and results within spatial branch-and-bound. Section~\ref{sec:conclusions} concludes the paper.

	\section{Sparse relaxations}
	\label{sec:sparse-rel}
	
	\subsection{A sparse reformulation of the SDP relaxation}\label{sec:sparsesdp}
	
	
	Recall the definition of the cone of PSD completable E-partial matrices from~\eqref{eq:psd-completable-matrices}.
	In the SDP relaxation~\eqref{eq:sdp}, all matrices $Q^k$, as well as the constraints defining $\mathcal Y$, are supported on $E$.
	Consequently, the objective and constraints depend on $Y$ only through $\Pi_E(Y)$.  
	Thus, as observed by \citet{fukudaetal2001} and \citet{burer2003}, retaining $Y$ as a full symmetric matrix, the SDP relaxation~\eqref{eq:sdp} is equivalently written as
	\begin{subequations}\label{eq:sdp-projected}
		\begin{align}
			z^{\mathrm{SDP}}
			= \min_{Y\in\mathcal Y}\quad
			& Q^0\bullet Y,\\
			\text{s.t.}\quad
			& Q^k\bullet Y\leq 0,
			\quad k=1,\ldots,m,\\
			& Y \in\Pi_E(\mathcal S^+).
		\end{align}
	\end{subequations}
	Here $Y$ technically remains a full matrix in $\mathcal S$, while the last
	constraint requires its entries indexed by $E$ to admit a positive
	semidefinite completion. The entries $Y_{ij}$ with $(i,j)\notin E$
	are unrestricted and do not affect the objective or any constraint.
	
	
	We stated the following well-known result (see, e.g., Chapter~1 in \cite{magron2023sparse}) that characterizes the dual cone of $\Pi_E(\mathcal{S}^+)$ in the introduction. We include a proof here for completeness.
	\begin{lemma}\label{thm:proj-esdp}
		The cones $\Pi_E(\mathcal{S}^+)$ and $\mathcal{S}^+ \cap \mathcal{S}(E)$ are dual cones in $\mathcal{S}(E)$.
	\end{lemma}
	\begin{proof}
		Let $\Pi_E(\mathcal{S}^+)^*$ denote the dual cone of $\Pi_E(\mathcal{S}^+)$ in $\mathcal{S}(E)$.
		
		First, suppose $Z \in \Pi_E(\mathcal{S}^+)$. Then,  there exists $Y \in \mathcal{S}^+$ with $Y_{ij} = Z_{ij}$ for all $(i,j) \in E$.  It holds that
		\[ \matip{C}{Z} = \matip{C}{Y} \geq 0 \quad \forall C \in \mathcal{S}^+ \cap \mathcal{S}(E),  \]
		where the first equality follows from $C \in \mathcal{S}(E)$ and the inequality follows because $Y \in \mathcal{S}^+$. This shows that $\Pi_E(\mathcal{S}^+)^* \supseteq \mathcal{S}^+ \cap \mathcal{S}(E)$.
		
		Next, let $C \in (\Pi_E(\mathcal{S}^+))^*$. Then $C \in \mathcal{S}(E)$ because $(\Pi_E(\mathcal{S}^+))^*$ is the dual cone of $\Pi_E(\mathcal{S})^+$ in $\mathcal{S}(E)$. Next, for any $Y \in \mathcal{S}^+$ it holds that
		\[ \matip{C}{Y} = \matip{C}{\Pi_E(Y)} \geq 0, \]
		where the equality follows because $C \in \mathcal{S}(E)$ and the inequality follows because $C \in (\Pi_E(\mathcal{S}^+))^*$. Thus $C \in (\mathcal{S}^+)^*=\mathcal{S}^+$, showing $(\Pi_E(\mathcal{S}^+))^* \subseteq \mathcal{S}^+ \cap \mathcal{S}(E)$.
	\end{proof}
	
	
	Lemma~\eqref{thm:proj-esdp} leads to the following ``sparse cut'' reformulation of \eqref{eq:sdp}:
	
	\begin{subequations}\label{SDP-E1}\begin{align}
			z^{\text{SDP-E}}=\min_{Y \in \mathcal{Y}} \ & \matip{Q^0}{Y}  \\
			\text{s.t. } & \matip{Q^k}{Y}  \leq 0, \quad k=1,\ldots, m \label{SDP-E1-b}\\
			& \matip{C}{Y} \geq 0, \quad \forall C \in \mathcal{S}^+ \cap \mathcal{S}(E) \label{eq:epsd} \\
			& Y_{ij} = 0, \quad \forall (i,j) \notin E. \label{redundant}
	\end{align}\end{subequations}
	We refer to the inequalities in \eqref{eq:epsd} as $E$-sparse PSD cuts, or simply sparse PSD cuts.
	The constraints \eqref{redundant} are technically redundant -- in an implementation the variables $Y_{ij}$ for $(i,j) \notin E$ would simply be omitted. We include them for consistency with Lemma \ref{thm:proj-esdp}.

	\begin{corollary} \label{thm-foo}
		It holds that $z^{\text{SDP}}=z^{\text{SDP-E}}$.
	\end{corollary}
	\begin{proof}
		First, it is immediate from \eqref{eq:semi-infinite-psd} that  $z^{\text{SDP}}\geq z^{\text{SDP-E}}$.
		
		Let $Z$ be a feasible solution to \eqref{SDP-E1}. By Lemma \ref{thm:proj-esdp}, $Z \in \Pi_E(\mathcal{S}^+)$ and hence there exists $Y \in \mathcal{S}^+$ such that  $Y_{ij} = Z_{ij}$ for all $(i,j) \in E$. Note that $Y$ satisfies \eqref{eq:sdp-cons} because $Q^k$ has nonzero entries only on $E$ and therefore
		$\matip{Q^k}{Z} = \matip{Q^k}{Y}$ for all $k=0,1,\ldots,m$.
		Thus,  $Y$ is feasible to \eqref{eq:sdp} and has the same objective value in \eqref{eq:sdp} as $Z$ has in \eqref{SDP-E1}. This establishes that  $z^{\text{SDP}}\leq z^{\text{SDP-E}}$. 
	\end{proof}

	\subsection{Replacing PSD constraints with finitely many linear inequalities via duality}
	\label{sec:replace_sdp_with_cuts_duality}
	
	Assume that $z^{\mathrm{SDP}}$ is finite and that the primal SDP~\eqref{eq:sdp} satisfies Slater's condition relative to the affine constraints defining $\mathcal Y$.
	Under these assumptions, conic strong duality holds and the dual optimal value is attained.  In this case, it is well-known that \citep{ramana1997exact,vandenberghe1996semidefinite,deroux.carr.ravi:25} strong duality holds and the SDP bound $z^{\text{SDP}}$ can be attained after replacing the SDP constraint \eqref{eq:semi-infinite-psd} with a small number of (up to $n+1$) linear inequalities of the form $\matip{v_t v_t^\top}{Y} \geq 0$,
	where the vectors $v_t$ are obtained from the rank-1 factorization of an optimal dual solution to the SDP relaxation~\eqref{eq:sdp}. 
	More precisely, let $S^*\succeq0$ denote the PSD slack matrix associated with an optimal dual solution. If $r=\operatorname{rank}(S^*)$, then $S^*$ admits a decomposition
	\[
	S^*=\sum_{t=1}^r v_tv_t^\top,
	\qquad r\leq n+1.
	\]
	It follows from the optimal dual certificate that the SDP optimal value can be reproduced by replacing the constraint $Y\succeq0$ with the $r$ linear inequalities
	\[
	\matip{v_tv_t^\top}{Y}=v_t^\top Yv_t\geq0, \qquad t=1,\ldots,r.
	\]
	Adding precisely these $r$ linear inequalities to the primal problem yields a \LP that achieves the SDP bound $z^{\text{SDP}}$, namely
	\begin{subequations}
		\label{eq:sdp2}
		\begin{align}
			z^{\text{SDP}}= \min_{Y \in \mathcal{Y}} \ & \matip{Q^0}{Y}  \\
			\text{s.t. } &\matip{Q^k}{Y}  \leq 0, \quad k=1,\ldots, m  \label{eq:sdp2-cons} \\
			& \matip{ (v_tv_t^\top) }{Y} \geq 0. \qquad t=1,\ldots, r  \label{eq:sdp2-dual-eigen}
		\end{align}
	\end{subequations}
	We next derive an upper bound on the number of inequalities that are sufficient to obtain the optimal value for the reformulation \eqref{SDP-E1}. Projecting out the $Y_{ij}$ variables for $(i,j) \not\in E$ from the formulation \eqref{eq:sdp2} yields the same optimal value.
	Recall that the constraints defining $\mathcal{Y}$ only operate on $Y_{ij}$ variables for $(i,j) \in E$. Thus, by applying Fourier-Motzkin elimination, there exists a finite set of vectors $V^1, \ldots V^\tau \in \mathcal{S}(E)$ such that 
	\begin{subequations}\label{eq:sdp2_E}\begin{align}
			z^{\text{SDP-E}}=\min_{Y \in \mathcal{Y}} \ & \matip{Q^0}{Y} \\
			\text{s.t. } & \matip{Q^k}{Y}  \leq 0,  \quad k=1,\ldots, m \\
			& \matip{V^t}{Y}  \ge 0
			,  \quad t=1,\ldots,\tau  \label{eq:sdp2-dual-eigen_E} \\
			& Y_{ij} = 0. \quad \forall (i,j) \notin E
	\end{align}\end{subequations}
		For notational simplicity, we write this as a formulation over variables $Y \in \mathcal{S}$, but after eliminating the variables fixed to $0$ this should be interpreted as an LP with $|E|$ variables.
		Thus, while  formulation \eqref{eq:sdp2_E} can have an exponential number of constraints of the form \eqref{eq:sdp2-dual-eigen_E}, at most $|E|$ many constraints from \eqref{eq:sdp2-dual-eigen_E} are sufficient to solve this LP. 
		
		\begin{lemma}
			\label{lem:evecsup}
			For $t=1,\ldots,\tau$, $V^t \in \mathcal{S}^+\cap \mathcal{S}(E)$. Thus, if Slater's condition holds for formulation \eqref{eq:sdp}, the bound $z^{\text{SDP-E}}$ in formulation \eqref{eq:sdp2} can be attained with at most $|E|$ $E$-sparse PSD cuts.
		\end{lemma}
		\begin{proof}
			Since the elements of the matrices $Q^k_{ij} = 0$ for $(i,j) \notin E$ for all $k=1,\ldots,m$, the inequalities \eqref{eq:sdp2-dual-eigen_E} obtained via Fourier-Motzkin elimination are a conic combination of the matrices  $\{(v_1v_1^T),\ldots,(v_rv_r^T)\}\subset\mathcal{S}^+$.
			This implies that $V^t \in \mathcal{S}^+$ for $t=1,\ldots,\tau$, completing the proof.  
		\end{proof}
		
		We have not established whether the upper bound in Lemma \ref{lem:evecsup} is tight - e.g., it is an open question whether $n+1$ linear inequalities may be sufficient as is the case for the dense cuts \eqref{eq:sdp2-dual-eigen}.
		
		
		
		
		\subsection{Sparse doubly nonnegative relaxation}
		
		When the set $\mathcal{X}$ defines sign restrictions $x \geq 0$, the PSD constraint \eqref{eq:sdp-psd} can be replaced with the requirement that $Y$ is doubly nonnegative (DNN), 
		$Y \in \mathcal{D} := \mathcal{S}^+ \cap \mathbb{R}_{\geq 0}^{(n+1)\times (n+1)}$. This leads to the DNN relaxation
		\begin{subequations}
			\label{eq:dnn}
			\begin{align}
				z^{\text{DNN}}= \min_{Y \in \mathcal{Y}} \ & \matip{Q^0}{Y}  \label{eq:dnn-obj}\\
				\text{s.t. } & \matip{Q^k}{Y} \leq 0, \quad k=1,\ldots, m \label{eq:dnn-cons} \\
				& Y \in \mathcal{D}.  \label{eq:dnn-psd}
			\end{align}
		\end{subequations}
		
		We now extend the analysis of Section \ref{sec:sparsesdp} to derive a sparse reformulation of \eqref{eq:dnn}. The dual cone of the doubly nonnegative cone $\mathcal{D}$ is the Minkowski sum 
		\[ \mathcal{D}^* = \mathcal{S}^+ + \mathbb{R}_{\geq 0}^{(n+1)\times (n+1)}, \]
		so that $Y \in \mathcal{D}$ can be equivalently written as $\matip{C}{Y} \geq 0~\text{for all } C \in \mathcal{D}^*.$
		As in the SDP case, the constraints \eqref{eq:dnn-psd} can be replaced with $Y \in \Pi_E(\mathcal{D})$ without changing problem \eqref{eq:dnn}. The following lemma is a direct analog of Lemma \ref{thm:proj-esdp} and is proved similarly.
		
		\begin{lemma}
			\label{thm:proj-dnn}
			The cones $\Pi_E(\mathcal{D})$ and $\mathcal{D}^* \cap \mathcal{S}(E)$ are dual cones in $\mathcal{S}(E)$. 
		\end{lemma}
		
		Thus, we obtain the following reformulation of \eqref{eq:dnn} based on sparse cuts:
		\begin{subequations}
			\label{eq:DNN-E}
			\begin{align}
				z^{\text{DNN-E}}=\min_{Y \in \mathcal{Y}} \ & \matip{Q^0}{Y} \\
				\text{s.t. } & \matip{Q^k}{Y}  \leq 0,  \quad k=1,\ldots, m \label{sdpe2-lin}\\
				& \matip{C}{Y} \geq 0 \quad \forall C \in \mathcal{D}^* \cap \mathcal{S}(E) \label{eq:DNN-E-psd} \\
				& Y_{ij} \geq 0 \quad \forall (i,j) \in E, \label{eq:sign} \\
				& Y_{ij} = 0 \quad \forall (i,j) \notin E.
		\end{align}\end{subequations}
		By Lemma \ref{thm:proj-dnn}, the sign constraints \eqref{eq:sign} are implied by \eqref{eq:DNN-E-psd}.  We include them explicitly since in our computational framework a relaxation of this problem is sequentially improved by adding inequalities from \eqref{eq:DNN-E-psd} as cuts and our approach for separating such cuts assumes the relaxation solution satisfies \eqref{eq:sign}.
		
		In this section, we have established that 
		$$z^{\text{QP}} ~\ge~z^{\text{DNN}}~=~z^{\text{DNN-E}}~\ge~z^{\text{SDP-E}}~=~ z^{\text{SDP}}.\label{eq:rev_ineq}$$

		
		\section{Solving the SDP-E and DNN-E relaxations}
	\label{sec:separation}
	We next describe the separation problem for finding violated sparse PSD cuts and discuss how to exploit nonnegativity constraints on the variables. We also describe our computational approach to accelerate the cutting plane procedure.
	
	
	\subsection{The separation problem}
	
	Assume we have solved a relaxation of \eqref{SDP-E1} where the constraints \eqref{eq:epsd} have not been enforced. Let $\widehat{Y}_{\mathrm{LP}} \in \mathcal S(E)$ be the current optimal solution. 
	Our goal is to identify a linear inequality among the omitted constraints \eqref{eq:epsd} that is violated by $\widehat{Y}_{\mathrm{LP}}$. To this end, we define the following SDP separation problem:
	\begin{equation}
		\label{eq:sdpsep}
		\min_{{C\in \mathcal S}} \Bigl\{ \matip{C}{\widehat{Y}_{\mathrm{LP}}}  : C \succeq 0, \ C_{ij} = 0 \ \forall(i,j) \notin E, \ \| \mathrm{diag}(C)\|_1 \leq 1 \Bigr\},
	\end{equation}
	where $\mathrm{diag}(C)$ denotes the vector obtained by the diagonal elements of $C$.
	%
	The constraint $\| \mathrm{diag}(C)\|_1 \leq 1$ provides a normalization among the cone of symmetric PSD matrices by requiring the nuclear norm to be at most one. This bounds the feasible region, but still allows efficiently finding a violated inequality defined by $C \in \mathcal{S}^+ \cap \mathcal{S}(E)$ whenever one exists. In fact, this separation problem finds the most violated inequality relative to the given normalization.
	
	\subsection{Exploiting nonnegativity}
	In case the QCQP \eqref{eq:qp-cons} includes the sign restriction $x\geq 0$, we want to solve the tighter 
	relaxation \eqref{eq:DNN-E} based on the inequalities \eqref{eq:DNN-E-psd}. 
	Thus, given a relaxation solution $\widehat{Y}_{\mathrm{LP}} \in \mathcal{S}(E)$, with $\widehat{Y}_{\mathrm{LP}} \geq 0$, we seek a formulation for the following problem of identifying a maximally violated (normalized) inequality:
	\begin{equation}
		\label{eq:dnnsep-init}
		\min_{C \in \mathcal S} \Bigl\{ \matip{C}{\widehat{Y}_{\mathrm{LP}}}  : C \in \mathcal{D}^*, \ C_{ij} = 0 \ \forall(i,j) \notin E, \ \| \mathrm{diag}(C) \|_1  \leq 1 \Bigr\}.
	\end{equation}
	
	\begin{lemma}
		\label{lem:dnnform}
		Problem \eqref{eq:dnnsep-init} is equivalent to problem 
		
		\begin{equation}
			\label{eq:dnnsep}
			\min_{{C\in \mathcal S}} \Bigl\{ \matip{C}{\widehat{Y}_{\mathrm{LP}}} : C \succeq 0, \ C_{ij} \leq 0 \ \forall(i,j) \notin E, \ \| \mathrm{diag}(C)\|_1 \leq 1 \Bigr\} .
		\end{equation}
	\end{lemma}
	\begin{proof}
		By definition, the following semidefinite optimization problem is a restatement of 
		\eqref{eq:dnnsep-init}, where we use $A$ in place of $C$ as the matrix defining the cut:
		\begin{align*}
			\min_{{A\in \mathcal S}} \ & \matip{A}{\widehat{Y}_{\mathrm{LP}}}  \\
			\text{s.t. } & A = C +D \\
			& D \geq 0 \\
			& C \succeq 0\\
			& A_{ij} = 0 \quad \text{for all } (i,j) \notin E\\
			& \| \mathrm{diag}(A) \|_1 \leq 1 .
		\end{align*}
		Since $\widehat{Y}_{\mathrm{LP}} \geq 0$,  there always exists an optimal solution with $D_{ij} = 0$ and $C_{ij} = A_{ij}$  for all $(i,j) \in E$.
		This also implies that $\mathrm{diag}(A)=\mathrm{diag}(C)$. Also note that $A_{ij} = 0$ implies that $C_{ij} = -D_{ij}\le 0$ for all $(i,j) \notin E$. 
		Recalling that $(j,j) \in E$ for $j=0,\ldots,n$, we can then eliminate the $A$ and $D$ variables to  simplify the separation problem to yield \eqref{eq:dnnsep}.
	\end{proof}
	Observe that the only difference between \eqref{eq:sdpsep} and \eqref{eq:dnnsep} is changing the constraints $C_{ij}=0$ to $C_{ij} \leq 0$ for $(i,j) \notin E$.
	

	
	\subsection{Accelerating the cutting plane procedure}
	
		
		The separation problems \eqref{eq:sdpsep} and \eqref{eq:dnnsep} can be used within a standard cutting plane procedure for iteratively improving a linear relaxation of problem \eqref{eq:nonconvex-qp}. Specifically, a linear relaxation of the form \eqref{SDP-E1} is solved but with the constraints \eqref{eq:epsd} omitted. Given the solution $\widehat{Y}_{\text{LP}} \in\mathcal S(E)$, either separation problem \eqref{eq:sdpsep} or \eqref{eq:dnnsep} is solved. If the objective is negative, the associated optimal matrix $C$ defines a linear inequality that is added to the relaxation and the process is continued until no more violated cuts are found or a stopping condition is met.
		
		To accelerate the closing of the SDP gap, we do not separate $\widehat Y_{\mathrm{LP}}$ directly. Let $Y_{\mathrm{SDP}}^\star \in \mathcal S^+$ be an optimal full-space solution of the SDP relaxation~\eqref{eq:sdp}. We separate the point
		\[ \widehat Y_\alpha := \alpha \widehat{Y}_{\mathrm{LP}} + (1-\alpha) \Pi_E(Y_{\mathrm{SDP}}^\star) \]
		for some $\alpha \in (0,1)$.

		
		Note that if the current master-LP bound has not yet reached the optimal SDP value, then 
		\[ Q^0 \bullet \widehat Y_\alpha = \alpha Q^0\bullet\widehat Y_{\mathrm{LP}} + (1-\alpha)z^{\mathrm{SDP}} < z^{\mathrm{SDP}}, \]
		so $\widehat Y_\alpha$ cannot satisfy the PSD-completion constraint. Hence, there exists $C\in\mathcal S^+\cap\mathcal S(E)$ such that
		\[ C \bullet \widehat Y_\alpha < 0. \]
		Moreover,
		\[ C \bullet \Pi_E(Y^*_{\mathrm{SDP}}) = C \bullet Y_{\mathrm{SDP}}^\star \geq 0, \]
		and therefore
		\[ C \bullet \widehat Y_{\mathrm{LP}} < 0. \]
		Thus, the inequality $C\bullet Y\geq0$ also separates the current LP solution.

		
		The hope is that by separating points close to an optimal SDP solution, the resulting inequalities better approximate the geometry of the PSD cone $\mathcal{S}^+$ near the optimum, thus accelerating convergence of the cutting plane procedure.  This approach has been used in the context of mixed integer programming by \citet{fischetti2001polyhedral}, but we believe its use in the context of cutting-plane methods for solving SDP is novel. The drawback of this procedure is that it requires first solving \eqref{eq:sdp} to obtain $Y_{\mathrm{SDP}}^\star$. However, this only needs to be done once and, if successful, may lead to a sparse linear relaxation that provides a bound nearly as strong as the SDP relaxation, which may be beneficial when used within a branch-and-bound procedure for solving \eqref{eq:nonconvex-qp}.
		
		
		
		
		\section{Computational experiments}
		\label{sec:computation}
		
		In this section, we report on experiments using the separation problems~\eqref{eq:sdpsep} and~\eqref{eq:dnnsep} in a cutting-plane procedure to obtain dual bounds  on instances of QCQP by (approximately) solving the relaxations~\eqref{SDP-E1} or~\eqref{eq:DNN-E}.  We also report on the computational behavior of using the sparse linear inequalities found by separation in a commercial solver to solve QCQP instances to optimality.
		
		Our test suite consists of two classes of QCQP instances: $(i)$ 126 instances of box-constrained-QP (BoxQCQP) generated as in~\cite{burer:10} and then augmented with nonconvex quadratic constraints,
		and $(ii)$ 110 instances from QPLIB \cite{furini2019qplib}. The BoxQCQP instances are synthetically generated to control sparsity, while instances from QPLIB may be extremely sparse or fully dense. The methodology for generating and selecting the BoxQCQP and QPLIB instances we use in our experiments is described in~\ref{A:instances}.  Different experiments are designed to test different algorithm features, so we do not use all instances in all experiments.
		Our computational environment for the experiments is described in~\ref{A:environment}.
		Detailed computational results can be found in  \ref{app:solvedBoxQCQP}--\ref{app:relaxboundDetails}.

		\subsection{Dual bound quality}
		\label{sec:bounds}
		
		Our first experiment is aimed at comparing the performance of the sparse PSD cuts found by our procedure against cuts obtained directly from the eigenvectors corresponding to negative eigenvalues in the full space of variables.  Since our final aim is to solve instances of QCQP to optimality, we also wish to understand the impact of the cutting planes in conjunction with other linear inequalities used by commercial solvers in relaxations for QCQP, such as the McCormick inequalities. If the original variables are bounded, say $Y_{0i} := x_i \in [\ell_i,u_i]$, then we may add \emph{McCormick inequalities} $Y_{ij} \in \mathcal{M}_{ij}$ to strengthen the bound given by~\eqref{SDP-E1} or~\eqref{eq:DNN-E}, namely
		\begin{multline}
			\mathcal{M}_{ij}
			:=
			\bigl\{(x_i,x_j,Y_{ij}) :
			Y_{ij} \geq \ell_j x_i + \ell_i x_j-\ell_i\ell_j,\quad
			Y_{ij} \geq u_j x_i + u_i x_j-u_i u_j,\\
			Y_{ij} \leq \ell_j x_i + u_i x_j-u_i\ell_j,\quad
			Y_{ij} \leq u_j x_i + \ell_i x_j-\ell_i u_j
			\bigr\}.
		\end{multline}
		Note that the linear inequalities can be written for $(i,j) \in E$ without destroying the problem's natural sparsity pattern. Also note that as we have $Y_{ij}=Y_{ji}$, it suffices to add McCormick inequalities only for $i>j$.

		Table \ref{tab:bounds} compares the following cutting plane approaches for obtaining dual bounds: 
		\begin{enumerate}
			\setlength{\itemsep}{-0.3ex}
			\item Dense PSD cuts, with $\mathcal{M}_{ij}$ added $\forall (i,j) \in [0,\ldots,n]^2$ (``Dense McC.+Cuts''), 
			\item Dense PSD cuts, with $\mathcal{M}_{ij}$ added $\forall (i,j) \in E$ (``Dense Cuts''), 
			\item Sparse PSD cuts, with $\mathcal{M}_{ij}$ inequalities added $\forall (i,j) \in E$ (``Sparse cuts''),
			\item Sparse PSD cuts generated via accelerated procedure, with $\mathcal{M}_{ij}$ inequalities added $\forall (i,j) \in E$ (``SDP+Sparse Cuts'').
		\end{enumerate}
		Separation of dense PSD cuts is done by adding the inequalities $v_t v_t^\top \bullet Y \geq 0$ for all eigenvectors $v_t$ with a negative eigenvalue in the spectral decomposition of the point $\widehat{Y}$ to be separated.
		Separation of sparse PSD cuts is done by solving the DNN-E separation problem \eqref{eq:dnnsep} if $x \geq 0$ and the SDP-E separation problem \eqref{eq:sdpsep} otherwise.

		For each method and instance class, Table~\ref{tab:bounds} reports the average of the following statistics: the number of separation rounds (\textit{iter}), the number of cuts added (\textit{cuts}), the average ``SDP-gap'' closed $GC(z^{\text{LP}})$, and the solution time of the final LP ($t_{\text{lastlp}}$).  (For ``Sparse Cuts'', the \textit{iter} column is omitted since exactly one cut is added per iteration).
		The SDP-gap is a relative measure of the bound obtained by the LP cutting plane method compared to the best achievable, namely
		\[ GC(z^{\text{LP}}) = (z^{\text{LP}} - z^{\text{McC}})/(z^{\text{SDP}} - z^{\text{McC}}), \]
		where $z^{\text{LP}}$ is the LP value at termination, $z^{\text{McC}}$ is the value of the LP with no PSD cuts added, and $z^{\text{SDP}}$ is the value obtained by solving the true SDP relaxation.
		All LPs are solved with \textsc{Gurobi} using the Barrier method.
		The cutting plane procedure is run until $GC(z^{\text{LP}}) > 0.99$ or a time limit of one hour is reached. 
		
		
		
		\begin{table}[ht]
			\centering
			\resizebox{\textwidth}{!}{
				\begin{tabular}{lc|rrrr|rrrr|rrr|rrrr}
					\toprule
					\multicolumn{2}{c|}{Instances} & \multicolumn{4}{c|}{Dense McC.+Cuts} & \multicolumn{4}{c|}{Dense Cuts} & \multicolumn{3}{c|}{Sparse Cuts} & \multicolumn{4}{c}{SDP+Sparse Cuts} \\
					Group & $\#$ & $iter$ & $cuts$ & $GC$ & $t_{lastlp}$ & $iter$ & $cuts$ & $GC$ & $t_{lastlp}$ & $cuts$ & $GC$ & $t_{lastlp}$ & $t_{SDP}$ & $cuts$ & $GC$ & $t_{lastlp}$ \\
					\midrule
					\multicolumn{15}{l}{\textbf{BoxQCQP}}\\
					\addlinespace[2pt]
					$n \in [20,90]$ & 58 & 95 & 2634 & 0.87 & 48.3 & 110 & 2965 & 0.50 & 71.0 & 480 & 0.97 & 0.7 & 11.2 & 13 & \textbf{0.98} & \textbf{0.1} \\
					$n \in [100, 150]$ & 36 & 45 & 2486 & 0.70 & 180.6 & 40 & 2291 & 0.42 & 239.9 & 197 & 0.79 & 0.8 & 143.1 & 14 & \textbf{0.99} & \textbf{0.1} \\
					$n \in [175, 200]$ & 24 & 18 & 1592 & 0.61 & 592.1 & 18 & 1648 & 0.38 & 545.6 & 15 & 0.53 & 0.3 & 961.5 & 8 & \textbf{0.97} & \textbf{0.2} \\
					\midrule
					\multicolumn{15}{l}{\textbf{QPLIB}}\\
					\addlinespace[2pt]
					$n \in [10, 100)$ & 31 & 112 & 2502 & 0.94 & 49.4 & 122 & 2738 & 0.93 & 47.9 & 990 & 0.88 & 9.6 & 10.5 & 478 & \textbf{0.95} & \textbf{7.1} \\
					$n \in [100, 150]$& 22 & 24 & 1548 & 0.83 & 192.4 & 26 & 1653 & 0.82 & 184.4 & 341 & 0.80 & 16.6 & 367.6 & 313 & \textbf{0.87} & \textbf{14.1} \\
					$n \in (150, 200]$ & 24 & 16 & 1414 & \textbf{0.65} & 657.9 & 17 & 1506 & 0.60 & 548.0 & 209 & 0.59 & 20.6 & 1885.1 & 125 & 0.62 & \textbf{10.7} \\
					\bottomrule
			\end{tabular}}
			\caption{Summary of Computational Performance for Cutting Plane Methods.}
		\label{tab:bounds}
	\end{table}

	A comparison of the first three blocks of Table~\ref{tab:bounds} demonstrates that sparse PSD cuts can be highly effective at improving the dual bound, particularly for BoxQCQP instances.  As the problem size grows, the dense and sparse cutting plane methods are able to do fewer iterations within the time limit, but for different reasons.  The dense cutting plane procedure spends the majority of its computing time solving the \LP relaxations at each iteration, as illustrated by the corresponding $t_{\text{lastlp}}$ value.  The majority of the computing time in the sparse cutting plane procedure is used solving the separation SDPs~\eqref{eq:sdpsep} or~\eqref{eq:dnnsep}.
	This computational behavior is also observed in Figure~\ref{fig:comparisonApproaches}, which depicts the SDP gap closed and the LP relaxation solution time as a function of time for the different methods on the BoxQCQP instance \texttt{125-025-1\_10qc}.
	We see that both ``sparse cuts'' and ``dense cuts + dense McCormick'' close the SDP gap at the same rate, but the LP relaxations for the sparse cuts solve \emph{much} faster.  (Note that the LP time is shown on a logarithmic scale).
	One goal of our cutting-plane procedure is to produce a PSD-enhanced LP relaxation that can be employed effectively in \BnBx. The sparse approach produces strong relaxations that solve \emph{orders of magnitude faster} than their dense counterparts. 
	

	\begin{figure}[hbt]
		\centering
		\includegraphics[width=1\textwidth]{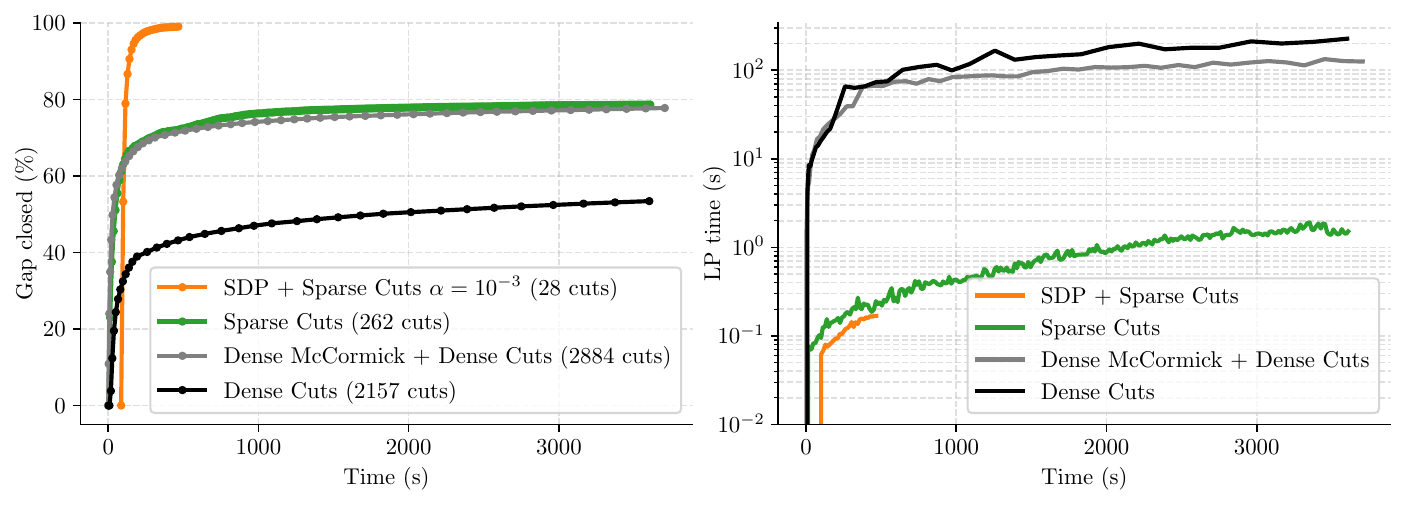}
		\vspace*{-5mm}
		\caption{Bound progression for \textsc{BoxQCQP} instance \texttt{125-025-1\_10qc}. Each dot represents an iteration. Note that the LP time is shown on a logarithmic scale.}
		\label{fig:comparisonApproaches}
	\end{figure}

	The final block of Table~\ref{tab:bounds} (``SDP+Sparse Cuts'') presents the results of our accelerated cut-generation procedure, which separates a point close to $\Pi_E(Y_{\mathrm{SDP}}^\star)$. The column $t_{\mathrm{SDP}}$ indicates the average time required to solve the initial SDP relaxation~\eqref{eq:sdp} and obtain $Y_{\mathrm{SDP}}^\star$.
	The results show that the SDP+Sparse Cuts variant closes nearly the entire gap for BoxQCQP instances and performs very effectively, often the best overall, on the QPLIB set.  The \LP relaxations remain very tractable, again an order of magnitude faster than their dense counterparts, making their use in \BnB practical. Figure \ref{fig:comparisonApproaches} also demonstrates  that the method quickly closes 99\% of the SDP gap.  Note that the orange curve is offset to the right to account for the time spent solving the SDP relaxation to obtain $Y_{\mathrm{SDP}}^\star$.

	
	\subsection{Sparse cutting planes with branch and bound} 
	\label{sec:Gurobi}

	In this section we compare the performance of \textsc{Gurobi} in its standalone configuration with \textsc{Gurobi} applied to instances augmented by the sparse cuts found by the ``SDP+Sparse Cuts'' procedure defined in Section~\ref{sec:bounds}. We omit instances where numerical errors were encountered (11 QPLIB instances), leaving 126 BoxQCQP and 99 QPLIB instances.
	For computing average solution times, runs that reach the 10-hour time limit are assigned a solution time of 36,000 seconds.
	
	For the BoxQCQP instances, the results are aggregated by size and by whether at least one of the two methods solved the instance to optimality within the time limit of ten hours.  For the QPLIB instances, the results are aggregated both by solvability of the instances and by two instance characteristics: density, either fully dense or not fully dense (sparse), and constraint type, either quadratic constraints are present (QC) or the constraints are all linear (LC).
	
	In Table~\ref{tab:Gurobi_aggregated}, 
	we report the average of the following  statistics for both methods: the gap closed at the root node ($GC_{\text{ro}}$), the number of explored nodes ($\textit{nodes}$), the final gap closed ($GC$), and the total solution time in seconds ($t$), as well as, in brackets, the number of instances solved to optimality within the time limit ($\#s$), where TO denotes that all instances timed-out.
	The gaps are computed by comparing the reported bound value $z$ to the bound of the natural McCormick relaxation $z^{\text{McC}}$ and normalizing by the largest possible difference: $(z - z^{\text{McC}})/(z^{\text{QP}} - z^{\text{McC}})$, where $z^{\text{QP}}$ is the optimal solution value, or best known solution value for unsolved instances.
	
	\begin{table}[h!]
		\centering
		\resizebox{\textwidth}{!}{%
			\begin{tabular}{lr|rrrr|rrrr}
				\toprule
				\multicolumn{2}{c|}{Instances} & \multicolumn{4}{c|}{\textsc{Gurobi}} & \multicolumn{4}{c}{SDP + Sparse Cuts + \textsc{Gurobi}} \\
				Group & $\#$ & $GC_{ro}$ & $nodes$ & $GC$ & $t$ (\#s) & $GC_{ro}$ & $nodes$ & $GC$ & $t$ (\#s) \\
				\midrule
				\multicolumn{10}{l}{\textbf{Solved BoxQCQP} (solved by at least one algorithm)}\\
				\addlinespace[2pt]
				$n\in[20, 90]$   & 66 & 0.72 & 6.7e3 & 0.99 & 1229\ns{65}  & 0.80 & 7.1e3 & \textbf{1.00} & \textbf{806}\ns{66} \\
				$n\in[100, 150]$ & 22 & 0.82 & 6.8e4 & 0.99 & 10822\ns{18} & 0.95 & 9.0e3 & \textbf{1.00} & \textbf{2126}\ns{22} \\
				$n\in[175, 200]$ &  1 & 0.77 & 3.3e4 & 0.81 & TO\ns{0} & 0.97 & 3.0e4 & \textbf{1.00} & \textbf{23533}\ns{1} \\
				\midrule
				\multicolumn{10}{l}{\textbf{Unsolved BoxQCQP} (not solved by any of the two algorithms)}\\
				\addlinespace[2pt]
				$n\in[100, 150]$ & 14 & 0.92 & 1.1e4 & 0.95 & TO & 0.96 & 1.6e4 & \textbf{0.98} & TO \\
				$n\in[175, 200]$ & 23 & 0.82 & 3.9e4 & 0.85 & TO & 0.95 & 3.5e4 & \textbf{0.96} & TO \\
				\midrule
				\multicolumn{10}{l}{\textbf{Solved QPLIB} ($n\leq 200$; solved by at least one algorithm)}\\
				\addlinespace[2pt]
				Fully Dense, QC & 12 & 0.97 & 1.6e3 & \textbf{1.00} & 4930\ns{12} & 0.98 & 1.2e3 & \textbf{1.00} & \textbf{4004}\ns{12} \\
				Sparse, QC      & 20 & 0.34 & 1.2e5 & 0.98 & 4314\ns{19} & 0.59 & 1.5e4 & \textbf{1.00} & \textbf{1763}\ns{20} \\
				Fully Dense, LC &  1 & 0.95 & 3.0e3 & \textbf{1.00} & 298\ns{1} & 0.98 & 4.5e3 & \textbf{1.00} & \textbf{260}\ns{1} \\
				Sparse, LC      & 55 & 0.72 & 3.4e5 & \textbf{1.00} & \textbf{602}\ns{55} & 0.80 & 1.7e5 & 0.98 & 13892\ns{39} \\
				\midrule
				\multicolumn{10}{l}{\textbf{Unsolved QPLIB} ($n\leq 200$; not solved by any of the two algorithms)}\\
				\addlinespace[2pt]
				Fully Dense, QC & 1 & 0.92 & 6.1e3 & \textbf{0.98} & TO & 0.93 & 6.1e3 & \textbf{0.98} & TO \\
				Sparse, QC      & 7 & 0.25 & 1.6e7 & \textbf{0.67} & TO & 0.25 & 1.7e7 & \textbf{0.67} & TO \\
				Sparse, LC      & 3 & 0.63 & 2.9e7 & \textbf{0.86} & TO & 0.63 & 1.4e7 & \textbf{0.86} & TO \\
				\bottomrule
		\end{tabular}}
		\caption{Comparison of \textsc{Gurobi} performance with and without SDP and sparse cuts. $\#s$: number of instances solved to optimality within the time limit. TO indicates that all runs in the group reached the 10-hour time limit.}
		\label{tab:Gurobi_aggregated}
	\end{table}
	
	The results in Table \ref{tab:Gurobi_aggregated} show clearly that \textsc{Gurobi} can benefit significantly from the addition of sparse PSD cuts.  In the case of BoxQCQP, \textsc{Gurobi} with cuts added solves instances faster, solves more instances, and closes more gap on unsolved instances.  For QPLIB, the picture is more complicated because the instances are more diverse. In general, adding sparse PSD cuts to the QPLIB instances either improves \textsc{Gurobi} performance or is performance neutral.  For the category in which performance degrades---sparse QPLIB instances with linear constraints---a closer inspection of the detailed results in Table~\ref{tableQPLIB_mod} of~\ref{app:solvedQPLIB} indicates that the tighter root bound of the formulation with cuts is not sufficient to recover the time spent solving the original SDP relaxation and the separation problems. In these instances, \textsc{Gurobi} alone can strengthen the bound efficiently through branching, and adding additional inequalities reduces the number of explored nodes within the time limit. 
	
	In Table~\ref{tab:Gurobi_aggregated_details}, we  report averages of the following detailed statistics for \textsc{Gurobi} with sparse cuts: the time required to solve the initial SDP ($t_{\text{SDP}}$), the gap closed by the SDP relaxation ($GC_{\text{SDP}}$), the time spent generating sparse cuts ($t_{\text{cuts}}$), the number of cuts (\textit{cuts}), and the gap closed by these cuts ($GC_{\text{cuts}}$). Although already reported in Table \ref{tab:Gurobi_aggregated}, to provide context for these results we again include in Table~\ref{tab:Gurobi_aggregated_details} the gap closed at the root node ($GC_{\text{ro}}$), the number the final gap closed ($GC$), and the total solution time in seconds ($t$).

	\begin{table}[h!]
		\centering
		\resizebox{\textwidth}{!}{%
			\begin{tabular}{lr|rcrrcccr}
				\toprule
				\multicolumn{2}{c|}{Instances} & \multicolumn{8}{c}{SDP + Sparse Cuts + \textsc{Gurobi}} \\
				Group & $\#$ & $t_{SDP}$ & $GC_{SDP}$ & $t_{cuts}$ & $cuts$ & $GC_{cuts}$ & $GC_{ro}$  & $GC$ & $t$ (\#s) \\
				\midrule
				\multicolumn{10}{l}{\textbf{Solved BoxQCQP} (solved by at least one algorithm)}\\
				\addlinespace[2pt]
				$n\in[20, 90]$   & 66 & 12 & 0.81 & 21 & 7 & 0.77 & 0.80 &  \textbf{1.00} & \textbf{806}\ns{66} \\
				$n\in[100, 150]$ & 22 & 146 & 0.95 & 183 & 4 & 0.92 & 0.95 &  \textbf{1.00} & \textbf{2126}\ns{22} \\
				$n\in[175, 200]$ & 1 & 1724 & 0.98 & 1934 & 2 & 0.95 & 0.97 &  \textbf{1.00} & \textbf{23533}\ns{1} \\
				\midrule
				\multicolumn{10}{l}{\textbf{Unsolved BoxQCQP} (not solved by any of the two algorithms)}\\
				\addlinespace[2pt]
				$n\in[100, 150]$ & 14 & 209 & 0.96 & 264 & 5 & 0.92 & 0.96 &\textbf{0.98} & TO \\
				$n\in[175, 200]$ & 23 & 1232 & 0.96 & 2127 & 5 & 0.92 & 0.95 &  \textbf{0.96} & TO \\
				\midrule
				\multicolumn{10}{l}{\textbf{Solved QPLIB} ($n\leq 200$; solved by at least one algorithm)}\\
				\addlinespace[2pt]
				Fully Dense, QC & 12 & 8 & 0.94 & 45 & 23 & 0.57 & 0.98 & \textbf{1.00} & \textbf{4004}\ns{12} \\
				Sparse, QC      & 20 & 312 & 0.46 & 68 & 10 & 0.39 & 0.59 & \textbf{1.00} & \textbf{1763}\ns{20} \\
				Fully Dense, LC & 1 & 8 & 0.42 & 25 & 25 & 0.14 & 0.98 &  \textbf{1.00} & \textbf{260}\ns{1} \\
				Sparse, LC      & 55 & 1134 & 0.85 & 748 & 57 & 0.47 & 0.80 &  0.98 & 13892\ns{39} \\
				\midrule
				\multicolumn{10}{l}{\textbf{Unsolved QPLIB} ($n\leq 200$; not solved by any of the two algorithms)}\\
				\addlinespace[2pt]
				Fully Dense, QC & 1 & 12 & 0.87 & 133 & 30 & 0.30 & 0.93 &  \textbf{0.98} & TO \\
				Sparse, QC      & 7 & 221 & 0.12 & 7 & 4 & 0.10 & 0.25 &  \textbf{0.67} & TO \\
				Sparse, LC      & 3 & 1344 & 0.66 & 145 & 24 & 0.56 & 0.63 &  \textbf{0.86} & TO \\
				\bottomrule
		\end{tabular}}
		\caption{Detailed results for SDP + Sparse Cuts + \textsc{Gurobi}.}
		\label{tab:Gurobi_aggregated_details}
	\end{table}
	
	
	We find that on the BoxQCQP instances the proposed method for generating sparse cuts closes the majority of SDP relaxation gap using a small number of cuts (less than 10 on average).  On the QPLIB instances the sparse cuts close significant gap but much less than the SDP relaxation and more cuts are added compared to the BoxQCQP instances. The gap closed after Gurobi processes the root is often much higher, suggesting our sparse cuts work effectively in conjunction with the other cuts used by Gurobi on these test instances. On all the instances, the time spent solving the SDP relaxation and generating the cuts is modest relative to the total computation time.

	\subsection{Comparing with dense SDP cuts and the objective cut} 
	\label{sec:SDP_Obj_Cut}
	
	Finally, to assess whether our approach is necessary, or whether a simpler mechanisms for transferring the SDP bound to the solver would suffice, we compare against two baselines that each sacrifice one property of our cuts.
	First, \emph{Dense Dual SDP Cuts} adds the rank-one inequalities
	$v_t v_t^{\top} \bullet Y \geq 0$ for all eigenvectors $v_t$ corresponding to
	positive eigenvalues of the dual optimal solution of the SDP relaxation. In
	contrast to the dense cuts of Section~\ref{sec:bounds}, these cuts are
	generated once from the dual optimum rather than by iterative separation. As
	discussed in Section~\ref{sec:replace_sdp_with_cuts_duality}, these cuts provably
	recover $z^{\text{SDP}}$ at the root relaxation, at the cost of being dense,
	i.e., not supported on the sparsity pattern $E$.
	Second, \emph{SDP Objective Bound Cut} instead adds the single inequality forcing the
	linearized objective to be at least $z^{\text{SDP}}$. This inequality is
	sparse, i.e., supported on $E$, but transfers only the bound itself and no
	further structural information about the PSD cone.
	

	Table~\ref{tab:sdpbound-baselines} reports the final gap closed and solution
	times; instance groups, the solved/unsolved partition, and all conventions are
	as in Table~\ref{tab:Gurobi_aggregated}. Note that the partition is
	defined with respect to \textsc{Gurobi} and SDP + Sparse Cuts only; the
	additional baselines occasionally solve instances in the unsolved groups.
	Both alternatives achieve the SDP bound at the root relaxation, yet neither
	is competitive: the dense cuts make the node relaxations prohibitively
	expensive, while the single bound inequality offers no structural guidance for
	branching. As a result, both are typically slower than standalone
	\textsc{Gurobi}, and both are dominated by the sparse cuts on every instance
	group. This indicates that the computational benefit of our approach stems
	from the sparse cuts themselves rather than from mere access to the SDP bound.
	
	\begin{table}[h]
		\centering
		\resizebox{\textwidth}{!}{
			\begin{tabular}{lr|cr|cr|cr|cr}
				\toprule
				\multicolumn{2}{c|}{Instances} & \multicolumn{2}{c|}{\textsc{Gurobi}} &
				\multicolumn{2}{c|}{\makecell{Sparse\\Cuts}} &
				\multicolumn{2}{c|}{\makecell{Dense Dual\\SDP Cuts}} &
				\multicolumn{2}{c}{\makecell{SDP Objective\\Bound Cut}} \\
				
				Group & $\#$ & $GC$ & $t$ (\#s) & $GC$ & $t$ (\#s) & $GC$ & $t$ (\#s) & $GC$ & $t$ (\#s) \\
				\midrule
				\multicolumn{10}{l}{\textbf{Solved BoxQCQP} (solved by \textsc{Gurobi} or SDP + Sparse Cuts)}\\
				\addlinespace[2pt]
				$n\in[20, 90]$   & 66 & 0.99 & 1229\ns{65}  & \textbf{1.00} & \textbf{806}\ns{66}   & 0.98 & 10491\ns{49} & 0.99 & 2935\ns{62} \\
				$n\in[100, 150]$ & 22 & 0.99 & 10822\ns{18} & \textbf{1.00} & \textbf{2126}\ns{22}  & 0.97 & 28648\ns{6}  & 0.98 & 18295\ns{12} \\
				$n\in[175, 200]$ &  1 & 0.81 & TO\ns{0}     & \textbf{1.00} & \textbf{23533}\ns{1}  & 0.98 & TO\ns{0}     & 0.98 & TO\ns{0} \\
				\midrule
				\multicolumn{10}{l}{\textbf{Unsolved BoxQCQP} (solved by neither)}\\
				\addlinespace[2pt]
				$n\in[100, 150]$ & 14 & 0.95 & TO & \textbf{0.98} & TO & 0.96 & TO & 0.96 & TO \\
				$n\in[175, 200]$ & 23 & 0.85 & TO & \textbf{0.96} & TO & \textbf{0.96} & TO & \textbf{0.96} & TO \\
				\midrule
				\multicolumn{10}{l}{\textbf{Solved QPLIB} ($n\leq 200$)}\\
				\addlinespace[2pt]
				Fully Dense, QC & 12 & \textbf{1.00} & 4930\ns{12} & \textbf{1.00} & \textbf{4004}\ns{12} & 0.99 & 7517\ns{10} & 0.98 & 18588\ns{6} \\
				Sparse, QC      & 20 & 0.98 & 4314\ns{19} & \textbf{1.00} & \textbf{1763}\ns{20} & 0.68 & 16352\ns{11} & 0.85 & 11993\ns{14} \\
				Fully Dense, LC &  1 & \textbf{1.00} & 298\ns{1} & \textbf{1.00} & 260\ns{1} & \textbf{1.00} & 207\ns{1} & \textbf{1.00} & \textbf{88}\ns{1} \\
				Sparse, LC      & 55 & \textbf{1.00} & \textbf{602}\ns{55} & 0.98 & 13892\ns{39} & 0.93 & 29476\ns{14} & 0.94 & 10287\ns{41} \\
				\midrule
				\multicolumn{10}{l}{\textbf{Unsolved QPLIB} ($n\leq 200$)}\\
				\addlinespace[2pt]
				Fully Dense, QC &  1 & \textbf{0.98} & TO & \textbf{0.98} & TO & 0.89 & TO & 0.87 & TO \\
				Sparse, QC      &  7 & \textbf{0.67} & TO & \textbf{0.67} & TO & 0.64 & 30939\ns{1} & 0.65 & 30975\ns{1} \\
				Sparse, LC      &  3 & 0.86 & TO & 0.86 & TO & \textbf{0.97} & TO & 0.66 & TO \\
				\bottomrule
		\end{tabular}}
		\caption{Comparison of standalone \textsc{Gurobi}, SDP + Sparse Cuts, Dense
			Dual SDP Cuts, and the SDP Objective Bound Cut on all instances. $GC$:
			final gap closed; $t$: mean solution time in seconds; $\#s$: number of
			instances solved to optimality within the time limit; all conventions as in
			Table~\ref{tab:Gurobi_aggregated}.}
		\label{tab:sdpbound-baselines}
	\end{table}

	\section{Conclusions}
	\label{sec:conclusions}
	
	We developed sparse linear inequalities to strengthening relaxations of nonconvex QCQPs. These inequalities are supported only on the lifted variables associated with the aggregate sparsity pattern of the quadratic functions. Using the duality between the PSD-completable cone and the sparse PSD cone, we showed that the resulting semi-infinite LP relaxation has the same optimal value as the Shor SDP relaxation. Under standard regularity assumptions, this value can be recovered using finitely many sparse PSD cuts. We also extended the approach to doubly nonnegative relaxations and formulated structured SDPs for separating the resulting inequalities.
	
	Our computational results demonstrate that these sparse cuts can recover much of the strength of the SDP relaxation while producing LP relaxations that remain inexpensive to reoptimize. When added before spatial branch-and-bound, the cuts substantially improve performance on the structured BoxQCQP instances and are beneficial for several classes of QPLIB instances, although the benefit is less uniform across that more heterogeneous collection. These results suggest that sparse PSD cuts provide a useful mechanism for incorporating the strength of semidefinite relaxations into general-purpose global optimization solvers. Future work includes reducing the cost of the initial SDP and separation computations, developing cut-selection strategies, and generating additional cuts dynamically within branch-and-bound.

\bibliographystyle{splncs04}\bibliography{sdp}

@article{burer.letchford:09,
author = "S. Burer and A. Letchford",
journal = "{SIAM} Journal on Optimization",
title = "On Nonconvex Quadratic Programming with Box Constriants",
volume = 20,
number = 2,
pages = "1073-1089",
year = 2009,
}

@article{dey.et.al:22,
author = "Santanu S. Dey and Aleksandr Kazachkov and Andrea Lodi and Gonzalo Mu{\~n}oz",
title = "Cutting Plane Generation through Sparse Principal Component Analysis",
journal = "SIAM Journal on Optimization",
volume = {32},
number = {2},
pages = {1319-1343},
year = {2022},
}

@article{bonami.gunluk.linderoth:18,
author = "P. Bonami and O. G{\"u}nl{\"u}k and J. Linderoth",
title = "Globally Solving Box-Constrained Nonconvex Quadratic Programming Problems with Box Constraints via Integer Programming Methods",
journal = "Mathematical Programming Computation",
year = 2018,
volume = 10,
pages = "333-382",
}

@article{mccormick:76,
author = "G. P. McCormick",
title = "Computability of Global Solutions to Factorable Nonconvex Programs: {P}art {I}---{Convex} Underestimating Problems",
journal = "Mathematical Programming",
volume = 10,
year = 1976,
pages = "147-175"
}

@article{anstreicher:09,
author = "K. M. Anstreicher", 
title = "Semidefinite Programming Versus the Reformulation-Linearization Technique for Nonconvex Quadratically Constrained Quadratic Programming", 
journal = "Journal of Global Optimization", 
volume = 43,
year = 2009, 
pages = "471-484"
}

@incollection{qualizza.belotti.margot:12,
author = "A. Qualizza and P. Belotti and F. Margot",
title = "Linear Programming Relaxations of Quadratically Constrained Quadratic Programs",
booktitle = "Mixed Integer Nonlinear Programming",
editor = "J. Lee and S. Leyffer",
year = 2012,
pages = "407-426",
publisher = "Springer"
}

@article{blekherman.et.al:22,
author = "Grigoriy Blekherman and Santanu S. Dey and Marco Molinaro and Shengding Sun",
title = "Sparse {PSD} approximation of the {PSD} cone",
journal = "Mathematical Programming",
year = 2022,
volume = 191,
pages = "981–1004",
}

@article{nakata.et.al:03,
author = "Kazuhide Nakata and Katsuki Fujisawa and Mituhiro Fukuda and Masakazu Kojima and Kazuo Murota",
title = "Exploiting sparsity in semidefinite programming via matrix completion {II:} implementation and numerical results",
journal = "Mathematical Programming", 
series = "B",
volume = 95,
pages = "303–327",
year = 2003,
}

@article{ramana1997exact,
title={An Exact Duality Theory for Semidefinite Programming and Its Complexity Implications},
author={Ramana, Motakuri V.},
journal={Mathematical Programming},
volume={77},
number={2},
pages={129--162},
year={1997},
publisher={Springer}
}

@article{vandenberghe1996semidefinite,
title={Semidefinite Programming},
author={Vandenberghe, Lieven and Boyd, Stephen},
journal={SIAM Review},
volume={38},
number={1},
pages={49--95},
year={1996},
publisher={Society for Industrial and Applied Mathematics}
}

@article{shor:90,
author = "N. Z. Shor",
title = "Dual Quadratic Estimates in Polynomial and Boolean Programming",
journal = "Annals of Operations Research",
volume = 25,
pages = "163-168", 
year = 1990
}

@article{fischetti2001polyhedral,
  title={A polyhedral approach to simplified crew scheduling and vehicle scheduling problems},
  author={Fischetti, Matteo and Lodi, Andrea and Martello, Silvano and Toth, Paolo},
  journal={Management Science},
  volume={47},
  number={6},
  pages={833--850},
  year={2001},
  publisher={INFORMS}
}

@article{furini2019qplib,
  title={QPLIB: a library of quadratic programming instances},
  author={Furini, Fabio and Traversi, Emiliano and Belotti, Pietro and Frangioni, Antonio and Gleixner, Ambros and Gould, Nick and Liberti, Leo and Lodi, Andrea and Misener, Ruth and Mittelmann, Hans and others},
  journal={Mathematical Programming Computation},
  volume={11},
  pages={237--265},
  year={2019},
  publisher={Springer Berlin Heidelberg}
}

@Article{harris2020numpy,
 title         = {Array programming with {NumPy}},
 author        = {Charles R. Harris and K. Jarrod Millman and St{\'{e}}fan J.
                 van der Walt and Ralf Gommers and Pauli Virtanen and David
                 Cournapeau and Eric Wieser and Julian Taylor and Sebastian
                 Berg and Nathaniel J. Smith and Robert Kern and Matti Picus
                 and Stephan Hoyer and Marten H. van Kerkwijk and Matthew
                 Brett and Allan Haldane and Jaime Fern{\'{a}}ndez del
                 R{\'{i}}o and Mark Wiebe and Pearu Peterson and Pierre
                 G{\'{e}}rard-Marchant and Kevin Sheppard and Tyler Reddy and
                 Warren Weckesser and Hameer Abbasi and Christoph Gohlke and
                 Travis E. Oliphant},
 year          = {2020},
 month         = sep,
 journal       = {Nature},
 volume        = {585},
 number        = {7825},
 pages         = {357--362},
 doi           = {10.1038/s41586-020-2649-2},
 publisher     = {Springer Science and Business Media {LLC}},
 url           = {https://doi.org/10.1038/s41586-020-2649-2}
}

@article{diamond2016cvxpy,
  author  = {Steven Diamond and Stephen Boyd},
  title   = {{CVXPY}: {A} {P}ython-embedded modeling language for convex optimization},
  journal = {Journal of Machine Learning Research},
  year    = {2016},
  volume  = {17},
  number  = {83},
  pages   = {1--5},
}

@misc{gurobi,
  author = {{Gurobi Optimization, LLC}},
  title = {{Gurobi Optimizer Reference Manual}},
  year = 2024,
  url = "https://www.gurobi.com"
}

@manual{mosek,
   author = "MOSEK ApS",
   title = "The MOSEK Python Fusion API manual. Version 11.0.",
   year = 2025,
   url = "https://docs.mosek.com/latest/pythonfusion/index.html"
}

@article{deroux.carr.ravi:25,
author = "Daniel de Roux and Robert Carr and R. Ravi",
title = "Instance-specific linear relaxations of semidefinite optimization problems",
journal = "Mathematical Programming Computation",
volume = 17, 
page = "385–435",
year = 2025
}

@unpublished{baltean-lugojan.et.al:19,
  author = {Radu Baltean-Lugojan and Pierre Bonami and Ruth Misener and Andrea Tramontani},
  note = {Optimization Online},
  title = {Scoring positive semidefinite cutting planes for quadratic
optimization via trained neural networks},
  year = {2019},
  url = "https://optimization-online.org/wp-content/uploads/2018/11/6943.pdf"
}

@INPROCEEDINGS{hijazi.coffrin.vanhentenryck:16,
  author={Hijazi, Hassan and Coffrin, Carleton and Van Hentenryck, Pascal},
  booktitle={2016 Power Systems Computation Conference (PSCC)}, 
  title={Polynomial SDP cuts for Optimal Power Flow}, 
  year={2016},
  volume={},
  number={},
  pages={1-7},
  doi={10.1109/PSCC.2016.7540908}}

@article{krishnan.mitchell:06,
author = {Kartik Krishnan and John E. Mitchell},
title = {A unifying framework for several cutting plane methods for semidefinite programming},
journal = {Optimization Methods and Software},
volume = {21},
number = {1},
pages = {57--74},
year = {2006},
publisher = {Taylor \& Francis},
doi = {10.1080/10556780500065283},
}

@article{billionnet.elloumi.lambert:12,
author = "Alain Billionnet and Sourour Elloumi and Am{\'e}lie Lambert",
title = "Extending the QCR method to the case of general mixed integer program",
journal = "Mathematical Programming", 
volume = 131,
number = 1,
pages = "381-401",
year = 2012,
}

@article{billionnet.elloumi:07,
author = "A. Billionnet and S. Elloumi", 
title = "Using a mixed integer quadratic programming solver for the unconstrained quadratic 0-1
problem",
journal = "Mathematical Programming",
volume = 109, 
pages = "55-68",
year = 2007,
}

@article{sun.et.al:19,
author = "D. Sun and K.-C. Toh and Y. Yuan and X.-Y. Zhao",
title = "{SDPNAL+: A M}atlab software for semidefinite programming with bound constraints (version 1.0)",
journal = "Optimization Methods and Software", 
pages = "1-29",
year = 2019,
 }

@article{malick.et.al:09,
author = "J. Malick and J. Povh and F. Rendl and A. Wiegele",
title = "Regularization methods for semidefinite programming",
journal = "SIAM Journal on Optimization",
volume = 20,
number = 1,
pages = "336-356",
year = 2009,
}

@book{nesterov.nemirovski:94,
author = "Yurii Nesterov  and Arkadii Nemirovskii",
title = "Interior-Point Polynomial Algorithms in Convex Programming",
publisher = "Society for Industrial and Applied Mathematics",
year = {1994},
}

@article{tawarmalani.sahinidis:05,
author = "M. Tawarmalani and N. V. Sahinidis",
title = "A Polyhedral Branch-and-Cut Approach to Global Optimization",
journal = "Mathematical Programming",
year = 2005,
volume = 103,
number = 2,
pages = "225-249",
}

@article{bestuzheva.et.al:25,
author = {K. Bestuzheva and A. Chmiela and B. M\"{u}ller and F. Serrano and S. Vigerske and F. Wegscheider},
title = "Global optimization of mixed-integer nonlinear programs with {SCIP} 8",
journal = "Journal of Global Optimization",
volume = 91,
number = 2,
pages = "287-310", 
year = 2025,
}

@article{majumdar.et.al:20,
author = "Anirudha Majumdar and Georgina and Amir Ali Ahmadi",
title = "A survey of recent scalability improvements for semidefinite programming with applications in machine learning",
journal = "Annual Reviews: Control, Robotics, and Autonomous Systems",
year = 2020,
volume = 3,
pages = "331-360",
}

@article{sherali.fraticelli:02,
author = "H. D. Sherali and B. M. P. Fraticelli", 
title = "Enhancing {RLT} relaxations via a new class of semidefinite cuts",
journal = "Journal of Global Optimization",
volume = 22, 
pages = "233-261", 
year = 2002
}

@article{kelley:60,
        AUTHOR = "J. E. Kelley",
        TITLE = "The cutting plane method for solving convex programs",
        JOURNAL = "Journal of {SIAM}",
        VOLUME = "8",
        NUMBER = "4", 
        PAGES = "703-712",
        YEAR = "1960"
}

@article{burer:10,
author = "S. Burer",
title = "Optimizing a Polyhedral-Semidefinite Relaxation of Completely Positive Programs",
journal = "Mathematical Programming Computation",
volume = 2,
number = 1,
pages = "1–19",
year = 2010,
}

@article{fukudaetal2001,
author = {Fukuda, Mituhiro and Kojima, Masakazu and Murota, Kazuo and Nakata, Kazuhide},
title = {Exploiting Sparsity in Semidefinite Programming via Matrix Completion I: General Framework},
journal = {SIAM Journal on Optimization},
volume = {11},
number = {3},
pages = {647-674},
year = {2001},
doi = {10.1137/S1052623400366218}
}

@article{burer2003,
author = {Burer, Samuel},
title = {Semidefinite Programming in the Space of Partial Positive Semidefinite Matrices},
journal = {SIAM Journal on Optimization},
volume = {14},
number = {1},
pages = {139-172},
year = {2003},
doi = {10.1137/S105262340240851X}
}

@book{magron2023sparse,
  title={Sparse polynomial optimization: theory and practice},
  author={Magron, Victor and Wang, Jie},
  year={2023},
  publisher={World Scientific}
}

@article{grone.et.al:84,
  author  = {Robert Grone and Charles R. Johnson and Eduardo M. S{\'a} and Henry Wolkowicz},
  title   = {Positive Definite Completions of Partial Hermitian Matrices},
  journal = {Linear Algebra and its Applications},
  volume  = {58},
  pages   = {109--124},
  year    = {1984},
  doi     = {10.1016/0024-3795(84)90207-6}
}

@article{agler.et.al:88,
  author  = {Jim Agler and J. William Helton and Scott McCullough and Leiba Rodman},
  title   = {Positive Semidefinite Matrices with a Given Sparsity Pattern},
  journal = {Linear Algebra and its Applications},
  volume  = {107},
  pages   = {101--149},
  year    = {1988},
  doi     = {10.1016/0024-3795(88)90240-6}
}

@article{vandenberghe.andersen:15,
  author  = {Lieven Vandenberghe and Martin S. Andersen},
  title   = {Chordal Graphs and Semidefinite Optimization},
  journal = {Foundations and Trends in Optimization},
  volume  = {1},
  number  = {4},
  pages   = {241--433},
  year    = {2015},
  doi     = {10.1561/2400000006}
}

@incollection{floudas.visweswaran:95,
  author    = {Christodoulos A. Floudas and V. Visweswaran},
  title     = {Quadratic Optimization},
  booktitle = {Handbook of Global Optimization},
  editor    = {Reiner Horst and Panos M. Pardalos},
  series    = {Nonconvex Optimization and Its Applications},
  volume    = {2},
  pages     = {217--269},
  publisher = {Springer},
  address   = {Boston, MA},
  year      = {1995},
}

@article{burer.letchford:12,
  author  = {Samuel Burer and Adam N. Letchford},
  title   = {Non-Convex Mixed-Integer Nonlinear Programming: A Survey},
  journal = {Surveys in Operations Research and Management Science},
  volume  = {17},
  number  = {2},
  pages   = {97--106},
  year    = {2012},
  doi     = {10.1016/j.sorms.2012.08.001}
}

@article{pezeshki.scharf.chong:10,
  author  = {Ali Pezeshki and Louis L. Scharf and Edwin K. P. Chong},
  title   = {The Geometry of Linearly and Quadratically Constrained Optimization Problems for Signal Processing and Communications},
  journal = {Journal of the Franklin Institute},
  volume  = {347},
  number  = {5},
  pages   = {818--835},
  year    = {2010},
  doi     = {10.1016/j.jfranklin.2010.03.005}
}

	\section*{Funding}
	Luedtke is supported by the Office of Naval Research grant N000142612067.
	
	
	\newpage
	\appendix
	\renewcommand{\thesection}{Appendix~\Alph{section}}
	\titleformat{\section}{\normalfont\Large\bfseries}{\thesection:}{1em}{}

	\section{Benchmark Instances}
	\label{A:instances}
	Our test suite comprises two classes of instances: (i) synthetic box-constrained QCQPs (BoxQCQPs) generated following the procedure in~\cite{burer:10}, and (ii) QCQP instances from the QPLIB~\cite{furini2019qplib}. Details for each class are provided below.

	

	\subsection*{BoxQCQP Instances}
	The first class consists of 126 box-constrained quadratic programs (BoxQCQPs) generated following the procedure in~\cite{burer:10}. Instances are generated for target densities $\rho \in \{0.1, 0.25\}$ and problem sizes $n \in [20,200]$; for $n\leq 60$, only the sparser versions ($\rho = 0.1$) are included. The nonzero entries of the objective coefficients $(\bar{Q}^0, c_0)$ are sampled uniformly from the integer interval $[-50, 50]$.  Unique preprocessing operations can be performed on box-constrained QPs that allow some variables to be treated as binary which can significantly improve computational performance \cite{bonami.gunluk.linderoth:18}.  We wish to focus on aspects related to the performance of PSD cuts, so we augment each problem with $k \in \{5, 10\}$ additional quadratic constraints, whose coefficients are sampled uniformly from $\mathrm{Unif}\{-50, \dots, 50\}$ on the support of the objective. To ensure feasibility, the right-hand side of each constraint is chosen such that the inequality holds with equality at the point $x = [0.5, \dots, 0.5]^{\mathsf T}$.  
	
	For the experiments measuring bound quality in Section~\ref{sec:bounds}, we removed two instances where the SDP bound was equal to the McCormick bound, and six instances where \textsc{Gurobi} or \textsc{Mosek} reported numerical errors in solving an LP or SDP, yielding 118 instances for that experiment.
	
	Instance names follow the pattern \texttt{spar[$n$]-[$\rho$]-[$i$]\_[$k$]qc}, where $n$ is the number of variables, $\rho$ the target density, $k$ the number of quadratic constraints, and $i \in \{1, 2, 3\}$ indexes the instance variation. 
	
	\subsection*{QPLIB Instances}
	The second class of test problems is drawn from the QPLIB benchmark set~\cite{furini2019qplib}, which provides a diverse collection of discrete and continuous quadratic programs. We restrict attention to instances with at most 200 variables to keep the SDP relaxation computationally tractable. Moreover, since the McCormick inequalities require bounded variables, we exclude 23 instances containing unbounded variables, which left 110 potential instances.  For the experiments measuring quality of the dual bounds in Section~\ref{sec:bounds}, we removed 16 instances where the SDP bound was equal to the McCormick bound, and 17 instances where \textsc{Gurobi} or \textsc{Mosek} reported numerical errors in solving an LP or SDP, yielding 77 QPLIB instances for testing bound quality.
	
	In contrast to the BoxQCQP set, the QPLIB instances range from extremely sparse to fully dense. Detailed instance statistics, including the number of variables and densities, are reported in Tables~\ref{tableQPLIB_mod} and~\ref{tableUnsolvedQPLIB}.

	
	
	\section{Computational Environment} 
	\label{A:environment}
	
	All experiments are conducted on a single thread of an \texttt{Intel Xeon Gold 6142} CPU running at 2.60~GHz, with 512~GB of RAM. The algorithms are implemented in Python 3.9.21, utilizing NumPy~2.0.2 \cite{harris2020numpy} for linear algebra computations. \SDPx s are solved using CVXPY~1.6.6 \cite{diamond2016cvxpy} with the Mosek~11.0.24 \cite{mosek} solver, while LPs and QCQPs are solved using the \textsc{Gurobi}~12.0.2 \cite{gurobi} optimizer. 
	
	Cuts are added if violated by more than $10^{-8}$. All solver tolerances are kept at their default values: $10^{-8}$ for both the \textsc{Mosek} SDP and the \textsc{Gurobi} barrier LP, and a relative optimality gap of $10^{-4}$ for branch-and-bound.
	
	\section{Solved BoxQCQP instances}
	\label{app:solvedBoxQCQP}
	
	Table~\ref{tab:solvedBQP} reports the BoxQCQP instances solved to optimality by at least one algorithm within the 10-hour time limit. For space reasons, the 64 instances that \textsc{Gurobi} solves in under 60 seconds are omitted. %
	The columns of the table are: 
	\begin{itemize}\itemsep -.1cm
		\item[-] \texttt{GC}$_{\mathrm{ro}}$ - relative gap closed at the root node; 
		\item[-] \texttt{nodes} - number of B\&B nodes explored; \item[-]\texttt{GC} - final relative gap closed; 
		\item[-]$t$ - total run time (in seconds); 
		\item[-]$t_{\mathrm{SDP}}$ - time spent solving the SDP relaxation; 
		\item[-] \texttt{GC}$_{\mathrm{sdp}}$ - relative gap closed by the SDP relaxation; 
		\item[-]$t_{\mathrm{cuts}}$ - time spent adding sparse cuts; 
		\item[-] \texttt{cuts} - number of sparse cuts; 
		\item[-] \texttt{GC}$_{\mathrm{cuts}}$ - relative gap closed after adding sparse cuts. 
	\end{itemize}
	All times are reported in seconds. 
	\begin{table}[h!]
		\centering
		\resizebox{\textwidth}{!}{\begin{tabular}{l|crcr|rcrrccrcr}
				\toprule
				Instance & \multicolumn{4}{c|}{\textsc{Gurobi}} & \multicolumn{9}{c}{SDP + Sparse Cuts + \textsc{Gurobi}} \\
				Name & $GC_{ro}$ & $nodes$ & $GC$ & $t~~$ & $t_{SDP}$ & $GC_{SDP}$ & $t_{cuts}$ & $cuts$ & $GC_{cuts}$ & $GC_{ro}$ & $nodes$ & $GC$ & $t~~$ \\
				\midrule
				070-025-1\_10qc & 0.86 & 4e+4 & \textbf{1.00} & \textbf{1965} & 13 & 0.90 & 5 & 3 & 0.88 & 0.90 & 1e+5 & \textbf{1.00} & 5668 \\
				070-025-2\_10qc & 0.85 & 4e+4 & \textbf{1.00} & 2098 & 13 & 0.92 & 6 & 4 & 0.89 & 0.92 & 4e+4 & \textbf{1.00} & \textbf{1942} \\
				070-025-2\_5qc & 0.90 & 3e+3 & \textbf{1.00} & 196 & 14 & 0.96 & 6 & 3 & 0.91 & 0.95 & 1e+3 & \textbf{1.00} & \textbf{60} \\
				080-025-1\_10qc & 0.88 & 5e+4 & \textbf{1.00} & \textbf{4439} & 20 & 0.90 & 15 & 4 & 0.87 & 0.90 & 1e+5 & \textbf{1.00} & 11795 \\
				080-025-1\_5qc & 0.91 & 2e+4 & \textbf{1.00} & 1651 & 17 & 0.93 & 11 & 5 & 0.90 & 0.93 & 7e+3 & \textbf{1.00} & \textbf{732} \\
				080-025-2\_10qc & 0.88 & 1e+5 & \textbf{1.00} & 13209 & 20 & 0.95 & 14 & 4 & 0.90 & 0.94 & 3e+4 & \textbf{1.00} & \textbf{6290} \\
				080-025-2\_5qc & 0.89 & 2e+4 & \textbf{1.00} & 2772 & 19 & 0.96 & 9 & 3 & 0.91 & 0.96 & 3e+3 & \textbf{1.00} & \textbf{222} \\
				090-025-1\_10qc & 0.91 & 8e+4 & 0.96 & TO & 33 & 0.93 & 28 & 5 & 0.89 & 0.93 & 6e+4 & \textbf{1.00} & \textbf{20657} \\
				090-025-1\_5qc & 0.94 & 3e+4 & \textbf{1.00} & 10544 & 24 & 0.96 & 13 & 4 & 0.92 & 0.96 & 2e+3 & \textbf{1.00} & \textbf{287} \\
				090-025-2\_10qc & 0.93 & 7e+3 & \textbf{1.00} & 3609 & 31 & 0.95 & 19 & 4 & 0.92 & 0.95 & 6e+3 & \textbf{1.00} & \textbf{1735} \\
				090-025-2\_5qc & 0.93 & 1e+4 & \textbf{1.00} & 4268 & 34 & 0.95 & 16 & 3 & 0.91 & 0.95 & 8e+3 & \textbf{1.00} & \textbf{1413} \\
				090-025-3\_10qc & 0.96 & 3e+2 & \textbf{1.00} & \textbf{79} & 35 & 0.96 & 20 & 4 & 0.92 & 0.96 & 4e+2 & \textbf{1.00} & 124 \\
				100-025-2\_5qc & 0.94 & 5e+3 & \textbf{1.00} & 5691 & 51 & 0.97 & 67 & 7 & 0.93 & 0.97 & 5e+2 & \textbf{1.00} & \textbf{348} \\
				100-025-3\_10qc & 0.93 & 3e+4 & 0.98 & TO & 43 & 0.96 & 35 & 5 & 0.93 & 0.96 & 2e+4 & \textbf{1.00} & \textbf{18274} \\
				100-025-3\_5qc & 0.97 & 8e+2 & \textbf{1.00} & 387 & 59 & 0.98 & 56 & 6 & 0.94 & 0.98 & 6e+2 & \textbf{1.00} & \textbf{323} \\
				125-010-1\_10qc & 0.70 & 3e+4 & \textbf{1.00} & 1423 & 134 & 0.93 & 185 & 3 & 0.89 & 0.92 & 8e+3 & \textbf{1.00} & \textbf{920} \\
				125-010-1\_5qc & 0.70 & 2e+4 & \textbf{1.00} & 787 & 102 & 0.94 & 132 & 3 & 0.91 & 0.93 & 2e+3 & \textbf{1.00} & \textbf{329} \\
				125-025-3\_5qc & 0.96 & 8e+3 & \textbf{1.00} & 31963 & 152 & 0.98 & 197 & 6 & 0.94 & 0.98 & 2e+3 & \textbf{1.00} & \textbf{2083} \\
				150-010-1\_10qc & 0.74 & 3e+5 & 0.90 & TO & 217 & 0.96 & 223 & 2 & 0.94 & 0.95 & 2e+4 & \textbf{1.00} & \textbf{3690} \\
				150-010-1\_5qc & 0.74 & 4e+5 & 0.92 & TO & 352 & 0.95 & 157 & 1 & 0.91 & 0.95 & 1e+5 & \textbf{1.00} & \textbf{10080} \\
				150-010-2\_10qc & 0.78 & 3e+5 & 0.98 & TO & 297 & 0.95 & 296 & 2 & 0.93 & 0.95 & 7e+3 & \textbf{1.00} & \textbf{2049} \\
				150-010-2\_5qc & 0.79 & 3e+5 & \textbf{1.00} & 28832 & 345 & 0.95 & 324 & 2 & 0.93 & 0.95 & 7e+3 & \textbf{1.00} & \textbf{1621} \\
				150-010-3\_10qc & 0.78 & 2e+5 & \textbf{1.00} & 23241 & 341 & 0.95 & 551 & 3 & 0.92 & 0.94 & 2e+4 & \textbf{1.00} & \textbf{3886} \\
				150-010-3\_5qc & 0.81 & 2e+4 & \textbf{1.00} & 1670 & 329 & 0.96 & 361 & 2 & 0.92 & 0.96 & 1e+3 & \textbf{1.00} & \textbf{787} \\
				200-010-3\_5qc & 0.77 & 3e+4 & 0.81 & TO & 1724 & 0.98 & 1934 & 2 & 0.95 & 0.97 & 3e+4 & \textbf{1.00} & \textbf{23533} \\
				
				\bottomrule
		\end{tabular}}
		\caption{BoxQCQP instances solved by at least one algorithm. For space
			reasons, only instances that \textsc{Gurobi} could not solve within 60 seconds
			are listed.}
		\label{tab:solvedBQP}
	\end{table}
	
	

	
	Figure~\ref{fig:BoxGQQP-bars} depicts the computing times for the 25 BoxQCQP instances solved to optimality by at least one method. Blue bars indicate the runtime of standalone \textsc{Gurobi}, while the multicolored bars represent the runtime breakdown of our approach: solving the SDP relaxation (black), separating the sparse cuts (gray), and solving the resulting QCQP with spatial \BnB using \textsc{Gurobi} (orange). The total height of each multicolored bar is the overall runtime of our method.

	{\begin{figure}[h!]
			\centering\includegraphics[width=.7\textwidth]{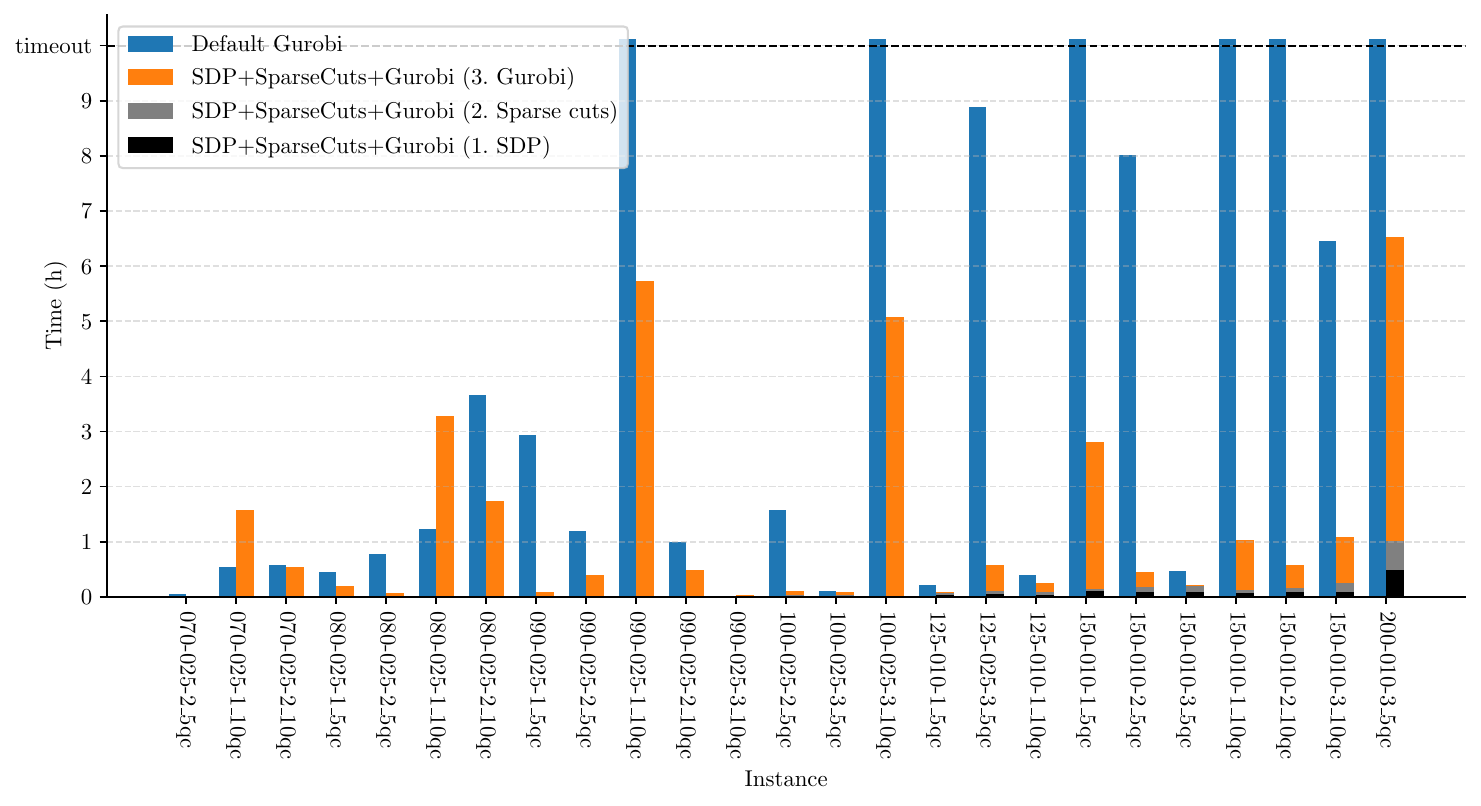}
			\caption{Computing time of the 25 \texttt{Solved BoxQCQP} instances in Table \ref{tab:solvedBQP}.}
			\label{fig:BoxGQQP-bars}
		\end{figure}
	}

\newpage.\newpage\section{Unsolved BoxQCQP instances}
	\label{app:unsolvedBoxQCQP}
	
	In Table \ref{tableUnsolvedBoxQP}, we report on the BoxQCQP instances solved by neither of the two methods within the 10-hour time limit. 

	\begin{table}[h]
		\centering\scriptsize
		\resizebox{\textwidth}{!}{\begin{tabular}{l|rccr|rcrrccrcr}
				\toprule
				Instance & \multicolumn{4}{c|}{\textsc{Gurobi}} & \multicolumn{9}{c}{SDP + Sparse Cuts + \textsc{Gurobi}} \\
				Name & $GC_{ro}$ & $nodes$ & $GC$ & $t~~$ & $t_{SDP}$ & $GC_{SDP}$ & $t_{cuts}$ & $cuts$ & $GC_{cuts}$ & $GC_{ro}$ & $nodes$ & $GC$ & $t~~$ \\
				\midrule
				100-025-1\_10qc & 0.89 & 3e+4 & 0.94 & TO & 49 & 0.94 & 44 & 5 & 0.89 & 0.94 & 5e+4 & \textbf{0.97} & TO \\
				100-025-1\_5qc & 0.90 & 4e+4 & 0.95 & TO & 40 & 0.95 & 32 & 5 & 0.90 & 0.94 & 5e+4 & \textbf{0.98} & TO \\
				100-025-2\_10qc & 0.89 & 3e+4 & 0.93 & TO & 56 & 0.94 & 60 & 6 & 0.90 & 0.93 & 5e+4 & \textbf{0.97} & TO \\
				125-025-1\_10qc & 0.89 & 7e+3 & 0.91 & TO & 145 & 0.95 & 175 & 6 & 0.90 & 0.94 & 9e+3 & \textbf{0.95} & TO \\
				125-025-1\_5qc & 0.90 & 8e+3 & 0.92 & TO & 116 & 0.95 & 143 & 6 & 0.91 & 0.94 & 1e+4 & \textbf{0.96} & TO \\
				125-025-2\_10qc & 0.95 & 6e+3 & 0.97 & TO & 136 & 0.97 & 150 & 6 & 0.93 & 0.97 & 8e+3 & \textbf{0.99} & TO \\
				125-025-2\_5qc & 0.95 & 7e+3 & 0.97 & TO & 138 & 0.98 & 151 & 6 & 0.93 & 0.97 & 9e+3 & \textbf{0.99} & TO \\
				125-025-3\_10qc & 0.94 & 8e+3 & 0.96 & TO & 146 & 0.96 & 171 & 6 & 0.92 & 0.96 & 1e+4 & \textbf{0.98} & TO \\
				150-025-1\_10qc & 0.93 & 3e+3 & 0.93 & TO & 268 & 0.97 & 369 & 6 & 0.92 & 0.97 & 5e+3 & \textbf{0.97} & TO \\
				150-025-1\_5qc & 0.94 & 3e+3 & 0.95 & TO & 380 & 0.98 & 584 & 7 & 0.94 & 0.97 & 4e+3 & \textbf{0.98} & TO \\
				150-025-2\_10qc & 0.93 & 3e+3 & 0.94 & TO & 353 & 0.97 & 422 & 5 & 0.92 & 0.96 & 5e+3 & \textbf{0.97} & TO \\
				150-025-2\_5qc & 0.94 & 3e+3 & 0.95 & TO & 353 & 0.97 & 409 & 5 & 0.93 & 0.97 & 5e+3 & \textbf{0.98} & TO \\
				150-025-3\_10qc & 0.95 & 3e+3 & 0.96 & TO & 373 & 0.98 & 475 & 6 & 0.93 & 0.98 & 4e+3 & \textbf{0.98} & TO \\
				150-025-3\_5qc & 0.95 & 3e+3 & 0.97 & TO & 366 & 0.98 & 516 & 7 & 0.94 & 0.98 & 5e+3 & \textbf{0.99} & TO \\
				175-010-1\_10qc & 0.70 & 1e+5 & 0.78 & TO & 767 & 0.95 & 831 & 2 & 0.93 & 0.95 & 6e+4 & \textbf{0.96} & TO \\
				175-010-1\_5qc & 0.72 & 1e+5 & 0.82 & TO & 616 & 0.96 & 706 & 2 & 0.94 & 0.95 & 1e+5 & \textbf{0.98} & TO \\
				175-010-2\_10qc & 0.73 & 1e+5 & 0.81 & TO & 732 & 0.94 & 718 & 2 & 0.92 & 0.94 & 7e+4 & \textbf{0.96} & TO \\
				175-010-2\_5qc & 0.77 & 9e+4 & 0.86 & TO & 816 & 0.96 & 841 & 2 & 0.93 & 0.96 & 1e+5 & \textbf{0.98} & TO \\
				175-010-3\_10qc & 0.69 & 1e+5 & 0.76 & TO & 796 & 0.93 & 820 & 2 & 0.91 & 0.92 & 6e+4 & \textbf{0.94} & TO \\
				175-010-3\_5qc & 0.69 & 2e+5 & 0.77 & TO & 707 & 0.93 & 1112 & 3 & 0.91 & 0.92 & 2e+5 & \textbf{0.94} & TO \\
				175-025-1\_10qc & 0.91 & 1e+3 & 0.91 & TO & 795 & 0.96 & 1684 & 8 & 0.92 & 0.96 & 3e+3 & \textbf{0.96} & TO \\
				175-025-1\_5qc & 0.93 & 1e+3 & 0.93 & TO & 882 & 0.97 & 1864 & 8 & 0.93 & 0.97 & 2e+3 & \textbf{0.97} & TO \\
				175-025-2\_10qc & 0.93 & 1e+3 & 0.94 & TO & 913 & 0.97 & 1611 & 7 & 0.93 & 0.97 & 2e+3 & \textbf{0.97} & TO \\
				175-025-2\_5qc & 0.93 & 1e+3 & 0.94 & TO & 881 & 0.97 & 1689 & 7 & 0.93 & 0.97 & 2e+3 & \textbf{0.97} & TO \\
				175-025-3\_10qc & 0.94 & 1e+3 & 0.94 & TO & 778 & 0.98 & 1607 & 8 & 0.93 & 0.97 & 2e+3 & \textbf{0.98} & TO \\
				175-025-3\_5qc & 0.94 & 1e+3 & 0.94 & TO & 912 & 0.98 & 1766 & 8 & 0.93 & 0.97 & 2e+3 & \textbf{0.98} & TO \\
				200-010-1\_10qc & 0.72 & 9e+3 & 0.75 & TO & 1865 & 0.94 & 1858 & 2 & 0.92 & 0.94 & 2e+4 & \textbf{0.95} & TO \\
				200-010-1\_5qc & 0.73 & 2e+4 & 0.78 & TO & 1568 & 0.95 & 1673 & 2 & 0.93 & 0.95 & 8e+4 & \textbf{0.96} & TO \\
				200-010-2\_10qc & 0.69 & 3e+4 & 0.73 & TO & 1643 & 0.95 & 1735 & 2 & 0.91 & 0.92 & 9e+3 & \textbf{0.93} & TO \\
				200-010-2\_5qc & 0.70 & 4e+4 & 0.76 & TO & 1682 & 0.96 & 1763 & 2 & 0.92 & 0.94 & 2e+4 & \textbf{0.96} & TO \\
				200-010-3\_10qc & 0.73 & 2e+4 & 0.75 & TO & 1403 & 0.95 & 1579 & 2 & 0.92 & 0.94 & 2e+4 & \textbf{0.95} & TO \\
				200-025-1\_10qc & 0.91 & 1e+3 & 0.91 & TO & 1790 & 0.97 & 4318 & 10 & 0.92 & 0.96 & 1e+3 & \textbf{0.96} & TO \\
				200-025-1\_5qc & 0.92 & 8e+2 & 0.92 & TO & 1812 & 0.97 & 4083 & 9 & 0.93 & 0.97 & 1e+3 & \textbf{0.97} & TO \\
				200-025-2\_10qc & 0.90 & 1e+3 & 0.90 & TO & 1686 & 0.95 & 4317 & 10 & 0.91 & 0.95 & 1e+3 & \textbf{0.95} & TO \\
				200-025-2\_5qc & 0.91 & 9e+2 & 0.91 & TO & 1762 & 0.96 & 4031 & 9 & 0.92 & 0.96 & 1e+3 & \textbf{0.96} & TO \\
				200-025-3\_10qc & 0.91 & 9e+2 & 0.91 & TO & 1801 & 0.97 & 4401 & 9 & 0.92 & 0.96 & 1e+3 & \textbf{0.96} & TO \\
				200-025-3\_5qc & 0.91 & 1e+3 & 0.91 & TO & 1729 & 0.97 & 3919 & 9 & 0.93 & 0.96 & 1e+3 & \textbf{0.97} & TO \\
				\bottomrule
			\end{tabular}
		}
		\caption{\texttt{BoxQCQP} instances solved by neither of the two methods.}\label{tableUnsolvedBoxQP}
	\end{table}
	

\newpage	\section{Solved QPLIB instances}
	\label{app:solvedQPLIB}
	
	In Table \ref{tableQPLIB_mod}, we report on the QPLIB instances solved by at least one algorithm within the 10-hour time limit. For space reasons, the 50 instances that \textsc{Gurobi} solves in under 60 seconds are omitted. 
	\begin{table}[h!]
		\centering
		\resizebox{\textwidth}{!}{
			\begin{tabular}{lrcc|crcr|rcrrccrcr}
				\toprule
				\multicolumn{4}{c|}{Instance} & \multicolumn{4}{c|}{Gurobi} & \multicolumn{9}{c}{SDP + Sparse Cuts + Gurobi} \\
				Name & n & $\rho$ & $\rho_{QC}$ & $GC_{ro}$ & nodes & $GC$ & $t$ & $t_{sdp}$ & $GC_{sdp}$ & $t_{cuts}$ & cuts & $GC_{cuts}$ & $GC_{ro}$ & nodes & $GC$ & $t$ \\
				\midrule
				\multicolumn{15}{l}{\textbf{Fully Dense, QC}}\\
				\addlinespace[2pt]
				QPLIB\_1055 & 40 & 1.00 & 0.52 & 0.98 & 5.6e2 & \textbf{1.00} & 102 & 4 & 0.98 & 4 & 20 & 0.62 & 0.99 & 2.4e2 & \textbf{1.00} & \textbf{55} \\
				QPLIB\_0911 & 50 & 1.00 & 0.26 & 0.96 & 8.6e2 & \textbf{1.00} & 1152 & 8 & 0.98 & 3 & 25 & 0.37 & 0.98 & 4.0e2 & \textbf{1.00} & \textbf{334} \\
				QPLIB\_0975 & 50 & 1.00 & 0.26 & 0.97 & 4.8e2 & \textbf{1.00} & 293 & 8 & 0.98 & 2 & 25 & 0.38 & 0.99 & 1.3e2 & \textbf{1.00} & \textbf{177} \\
				QPLIB\_1745 & 50 & 1.00 & 0.25 & 0.97 & 9.7e1 & \textbf{1.00} & \textbf{114} & 9 & 0.82 & 69 & 25 & 0.39 & 0.97 & 1.2e2 & \textbf{1.00} & 248 \\
				QPLIB\_1967 & 50 & 1.00 & 0.51 & 0.96 & 1.1e4 & \textbf{1.00} & 24075 & 9 & 0.96 & 5 & 25 & 0.64 & 0.97 & 8.9e3 & \textbf{1.00} & \textbf{19256} \\
				QPLIB\_1535 & 60 & 1.00 & 0.49 & 0.94 & 4.0e3 & \textbf{1.00} & 27236 & 17 & 0.91 & 200 & 30 & 0.55 & 0.95 & 3.1e3 & \textbf{1.00} & \textbf{24058} \\
				QPLIB\_1703 & 60 & 1.00 & 0.50 & 0.95 & 1.4e3 & \textbf{1.00} & 6044 & 14 & 0.95 & 154 & 30 & 0.62 & 0.98 & 8.8e2 & \textbf{1.00} & \textbf{3650} \\
				\midrule
				\multicolumn{15}{l}{\textbf{Sparse, QC}}\\
				\addlinespace[2pt]
				QPLIB\_1922 & 30 & 0.48 & 0.25 & 0.00 & 2.0e4 & \textbf{1.00} & 111 & 1 & 0.89 & 0 & 7 & 0.84 & 0.84 & 1.1e3 & \textbf{1.00} & \textbf{17} \\
				QPLIB\_1931 & 40 & 0.48 & 0.25 & 0.00 & 4.8e5 & \textbf{1.00} & 4038 & 3 & 0.98 & 1 & 9 & 0.94 & 0.94 & 1.9e2 & \textbf{1.00} & \textbf{9} \\
				QPLIB\_1913 & 48 & 0.48 & 0.13 & 0.00 & 4.3e4 & \textbf{1.00} & 1879 & 5 & 0.94 & 6 & 23 & 0.89 & 0.89 & 1.2e3 & \textbf{1.00} & \textbf{124} \\
				QPLIB\_1940 & 48 & 0.48 & 0.13 & 0.00 & 7.5e4 & \textbf{1.00} & 4165 & 5 & 0.92 & 7 & 24 & 0.87 & 0.87 & 7.7e2 & \textbf{1.00} & \textbf{105} \\
				QPLIB\_1437 & 50 & 0.99 & 0.49 & 0.84 & 3.5e2 & \textbf{1.00} & 162 & 8 & 0.84 & 46 & 25 & 0.57 & 0.99 & 3.9e1 & \textbf{1.00} & \textbf{107} \\
				QPLIB\_1886 & 50 & 0.49 & 0.25 & 0.00 & 1.4e6 & 0.64 & TO & 6 & 0.95 & 4 & 12 & 0.90 & 0.90 & 1.2e4 & \textbf{1.00} & \textbf{1136} \\
				QPLIB\_1661 & 60 & 0.99 & 0.49 & 0.85 & 3.7e2 & \textbf{1.00} & \textbf{553} & 14 & 0.89 & 120 & 30 & 0.67 & 0.99 & 1.9e2 & \textbf{1.00} & 929 \\
				QPLIB\_1773 & 60 & 0.99 & 0.49 & 0.87 & 6.0e3 & \textbf{1.00} & 35341 & 14 & 0.93 & 101 & 30 & 0.54 & 0.96 & 3.5e3 & \textbf{1.00} & \textbf{21565} \\
				QPLIB\_3562 & 63 & 0.02 & 0.01 & 0.89 & 1.7e5 & \textbf{1.00} & \textbf{196} & 7 & 0.00 & 0 & 0 & 0.00 & 0.89 & 1.7e5 & \textbf{1.00} & 204 \\
				QPLIB\_2055 & 153 & 0.20 & 0.01 & 0.03 & 3.1e3 & \textbf{1.00} & \textbf{527} & 316 & 0.00 & 0 & 0 & 0.00 & 0.03 & 3.1e3 & \textbf{1.00} & 843 \\
				QPLIB\_2060 & 171 & 0.19 & 0.01 & 0.05 & 9.1e3 & \textbf{1.00} & \textbf{2482} & 575 & 0.00 & 0 & 0 & 0.00 & 0.05 & 9.1e3 & \textbf{1.00} & 3074 \\
				QPLIB\_2067 & 190 & 0.18 & 0.01 & 0.01 & 1.5e3 & \textbf{1.00} & \textbf{564} & 959 & 0.00 & 0 & 0 & 0.00 & 0.01 & 1.5e3 & \textbf{1.00} & 1530 \\
				QPLIB\_3512 & 191 & 0.00 & 0.01 & 0.18 & 4.0e4 & \textbf{1.00} & \textbf{129} & 1477 & 0.00 & 0 & 0 & 0.00 & 0.18 & 4.0e4 & \textbf{1.00} & 1603 \\
				\midrule
				\multicolumn{15}{l}{\textbf{Fully Dense, LC}}\\
				\addlinespace[2pt]
				QPLIB\_0343 & 50 & 1.00 & 0.00 & 0.95 & 3.0e3 & \textbf{1.00} & 298 & 8 & 0.42 & 25 & 25 & 0.14 & 0.98 & 4.5e3 & \textbf{1.00} & \textbf{260} \\
				\midrule
				\multicolumn{15}{l}{\textbf{Sparse, LC}}\\
				\addlinespace[2pt]
				QPLIB\_0633 & 75 & 0.97 & 0.00 & 0.80 & 2.0e6 & \textbf{1.00} & \textbf{749} & 27 & 0.99 & 314 & 38 & 0.58 & 0.90 & 9.9e2 & \textbf{1.00} & 1459 \\
				QPLIB\_0067 & 80 & 0.87 & 0.00 & 0.00 & 5.7e5 & \textbf{1.00} & \textbf{135} & 31 & 0.00 & 0 & 0 & 0.00 & 0.00 & 5.7e5 & \textbf{1.00} & 166 \\
				QPLIB\_2512 & 100 & 0.76 & 0.00 & 0.00 & 4.1e5 & \textbf{1.00} & \textbf{120} & 63 & 0.00 & 0 & 0 & 0.00 & 0.00 & 4.1e5 & \textbf{1.00} & 185 \\
				QPLIB\_5935 & 100 & 0.98 & 0.00 & 0.72 & 1.5e5 & \textbf{1.00} & \textbf{183} & 66 & 0.32 & 21 & 50 & 0.24 & 0.94 & 3.6e3 & \textbf{1.00} & 5249 \\
				QPLIB\_5881 & 120 & 0.29 & 0.00 & 0.89 & 2.7e3 & \textbf{1.00} & \textbf{337} & 105 & 0.97 & 125 & 8 & 0.92 & 0.97 & 1.9e3 & \textbf{1.00} & 1067 \\
				QPLIB\_3402 & 144 & 0.80 & 0.00 & 0.00 & 8.3e6 & \textbf{1.00} & \textbf{3872} & 352 & 0.00 & 0 & 0 & 0.00 & 0.00 & 8.3e6 & \textbf{1.00} & 4411 \\
				QPLIB\_3751 & 150 & 0.32 & 0.00 & 0.36 & 3.2e3 & \textbf{1.00} & \textbf{482} & 330 & 0.98 & 684 & 17 & 0.81 & 0.98 & 2.7e1 & \textbf{1.00} & 1058 \\
				QPLIB\_5962 & 150 & 0.98 & 0.00 & 0.73 & 2.5e6 & \textbf{1.00} & \textbf{6765} & 307 & 0.32 & 106 & 75 & 0.24 & 0.91 & 9.0e2 & 0.95 & TO \\
				QPLIB\_5971 & 150 & 0.98 & 0.00 & 0.89 & 3.1e4 & \textbf{1.00} & \textbf{87} & 316 & 0.31 & 106 & 75 & 0.23 & 0.97 & 2.1e3 & \textbf{1.00} & 8340 \\
				QPLIB\_5980 & 150 & 0.98 & 0.00 & 0.99 & 4.8e2 & \textbf{1.00} & \textbf{78} & 351 & 0.31 & 119 & 75 & 0.23 & 0.99 & 3.4e2 & \textbf{1.00} & 767 \\
				QPLIB\_10048 & 150 & 0.98 & 0.00 & 0.97 & 3.1e5 & \textbf{1.00} & \textbf{123} & 661 & 0.97 & 133 & 75 & 0.39 & 0.94 & 6.3e2 & 0.95 & TO \\
				QPLIB\_3775 & 180 & 0.32 & 0.00 & 0.00 & 3.6e6 & \textbf{1.00} & \textbf{2089} & 1236 & 0.98 & 1832 & 16 & 0.82 & 0.99 & 1.1e1 & \textbf{1.00} & 3166 \\
				QPLIB\_3523 & 182 & 0.13 & 0.00 & 0.27 & 4.8e3 & \textbf{1.00} & \textbf{844} & 935 & 0.83 & 6534 & 20 & 0.50 & 0.71 & 9.3e3 & \textbf{1.00} & 17888 \\
				QPLIB\_3870 & 182 & 0.22 & 0.00 & 0.10 & 1.1e4 & \textbf{1.00} & \textbf{8517} & 1012 & 0.81 & 6215 & 26 & 0.44 & 0.51 & 1.2e4 & 0.86 & TO \\
				QPLIB\_3883 & 182 & 0.17 & 0.00 & 0.26 & 9.0e3 & \textbf{1.00} & \textbf{7943} & 935 & 0.80 & 6374 & 25 & 0.48 & 0.55 & 9.0e3 & 0.86 & TO \\
				QPLIB\_3803 & 190 & 0.13 & 0.00 & 0.47 & 1.4e3 & \textbf{1.00} & \textbf{367} & 1242 & 0.79 & 6120 & 14 & 0.10 & 0.81 & 2.4e2 & \textbf{1.00} & 7695 \\
				QPLIB\_10058 & 200 & 0.88 & 0.00 & 0.97 & 4.4e5 & \textbf{1.00} & \textbf{242} & 3718 & 0.97 & 55 & 10 & 0.00 & 0.98 & 7.1e1 & \textbf{1.00} & 6691 \\
				\bottomrule
		\end{tabular}}
		\caption{QPLIB instances solved by at least one algorithm within the time
			limit of 10h. For space reasons, only instances that \textsc{Gurobi} could not
			solve within 60 seconds are listed.}
		\label{tableQPLIB_mod}
	\end{table}
	

	In Figure \ref{instanceRuntimeBarsQPLib}, we show the runtimes for each component of our approach and compare it with \textsc{Gurobi} on the QPLIB instances of Table~\ref{tableQPLIB_mod}. 
	
	\begin{figure}[h!]
		\centering
		\includegraphics[width=.81\textwidth]{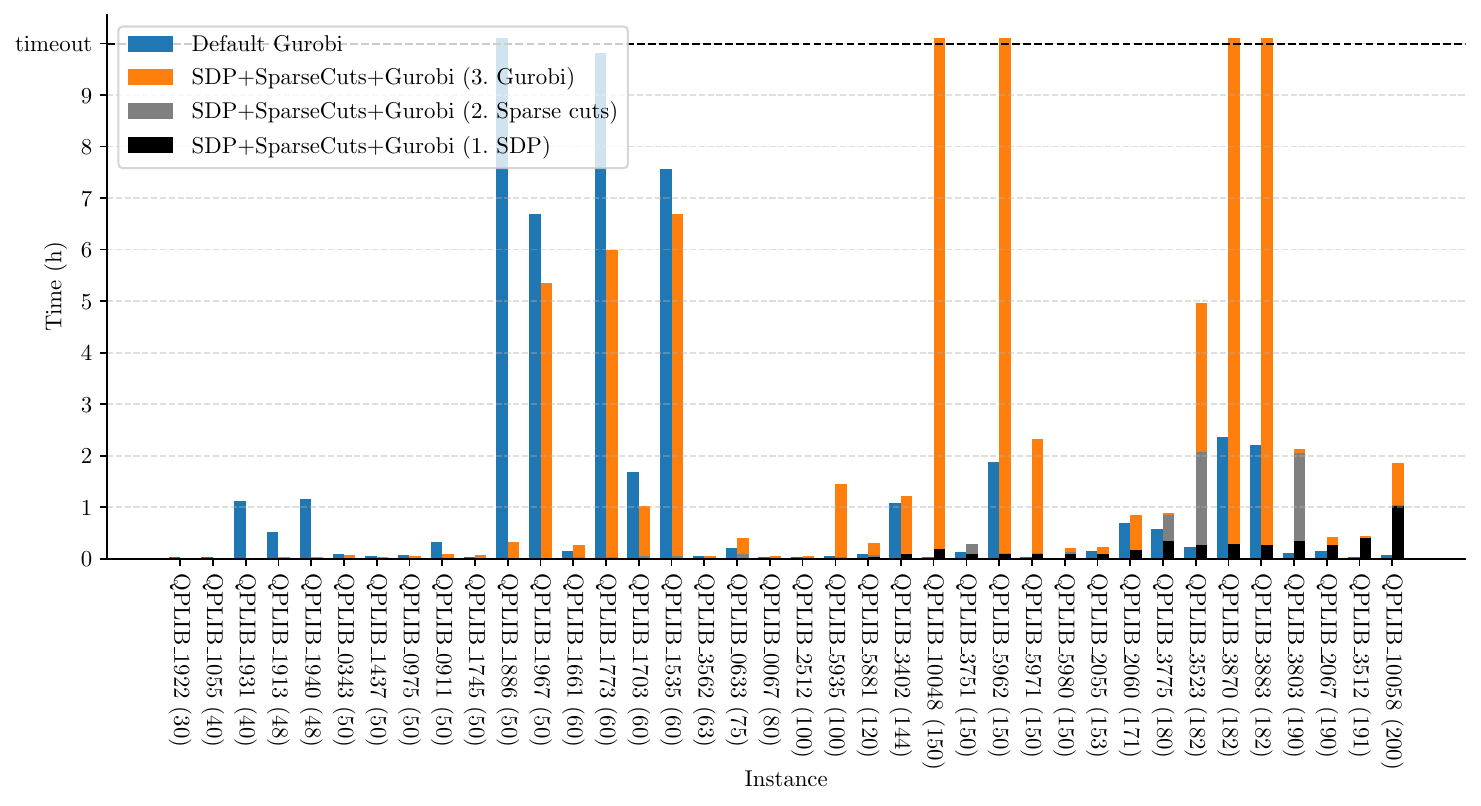}
		\caption{QPLIB instances solved by at least one algorithm within the time limit of 10h.}\label{instanceRuntimeBarsQPLib}
	\end{figure}
	
	\newpage\section{Unsolved QPLIB instances}
	\label{app:unsolvedQPLIB}
	
	In Table \ref{tableUnsolvedQPLIB}, we report on the QPLIB instances solved by neither of the two methods within the 10-hour time limit.
	
	\begin{table}[h!]
		\centering
		\resizebox{\textwidth}{!}{
			\begin{tabular}{lrcc|rccr|rcrrccrcr}
				\toprule
				\multicolumn{4}{c|}{Instance} & \multicolumn{4}{c|}{Gurobi} & \multicolumn{9}{c}{SDP + Sparse Cuts + Gurobi} \\
				Name & n & $\rho$ & $\rho_{QC}$ & $GC_{ro}$ & nodes & GC & $t$ & $t_{sdp}$ & $GC_{sdp}$ & $t_{cuts}$ & cuts & $GC_{cuts}$ & $GC_{ro}$ & nodes & GC & $t$ \\
				\midrule
				\multicolumn{15}{l}{\textbf{Fully Dense, QC}}\\
				\addlinespace[2pt]
				QPLIB\_1451 & 60 & 1.00 & 0.25 & 0.92 & 6.1e3 & \textbf{0.98} & TO & 12 & 0.87 & 133 & 30 & 0.30 & 0.93 & 6.1e3 & \textbf{0.98} & TO \\
				\midrule
				\multicolumn{15}{l}{\textbf{Sparse, QC}}\\
				\addlinespace[2pt]
				QPLIB\_2416 & 25 & 0.79 & 0.05 & 0.00 & 1.1e7 & \textbf{0.63} & TO & 1 & 0.00 & 0 & 0 & 0.00 & 0.00 & 1.1e7 & \textbf{0.63} & TO \\
				QPLIB\_2590 & 25 & 0.79 & 0.05 & 0.00 & 1.1e7 & \textbf{0.65} & TO & 1 & 0.00 & 0 & 0 & 0.00 & 0.00 & 1.1e7 & \textbf{0.65} & TO \\
				QPLIB\_3318 & 25 & 0.79 & 0.05 & 0.00 & 1.2e7 & \textbf{0.61} & TO & 1 & 0.00 & 0 & 0 & 0.00 & 0.00 & 1.2e7 & \textbf{0.61} & TO \\
				QPLIB\_1675 & 60 & 0.73 & 0.25 & 0.87 & 2.3e4 & \textbf{0.98} & TO & 11 & 0.88 & 53 & 30 & 0.75 & 0.92 & 2.1e4 & \textbf{0.98} & TO \\
				QPLIB\_3358 & 158 & 0.00 & 0.01 & 0.00 & 3.6e7 & \textbf{0.00} & TO & 424 & 0.00 & 0 & 0 & 0.00 & 0.00 & 3.6e7 & \textbf{0.00} & TO \\
				QPLIB\_3780 & 168 & 0.01 & 0.01 & 0.88 & 9.2e6 & \textbf{0.89} & TO & 747 & 0.00 & 0 & 0 & 0.00 & 0.88 & 9.6e6 & \textbf{0.89} & TO \\
				QPLIB\_2658 & 184 & 0.00 & 0.01 & 0.00 & 3.6e7 & \textbf{0.92} & TO & 359 & 0.00 & 0 & 0 & 0.00 & 0.00 & 3.6e7 & \textbf{0.92} & TO \\
				\midrule
				\multicolumn{15}{l}{\textbf{Sparse, LC}}\\
				\addlinespace[2pt]
				QPLIB\_5882 & 150 & 0.77 & 0.00 & 0.94 & 8.3e6 & 0.94 & TO & 344 & 0.98 & 104 & 37 & 0.84 & 0.95 & 1.1e3 & \textbf{0.97} & TO \\
				QPLIB\_2492 & 196 & 0.85 & 0.00 & 0.00 & 3.9e7 & \textbf{0.66} & TO & 2037 & 0.00 & 1 & 0 & 0.00 & 0.00 & 4.1e7 & \textbf{0.66} & TO \\
				QPLIB\_5875 & 200 & 0.78 & 0.00 & 0.95 & 3.8e7 & \textbf{0.97} & TO & 1650 & 0.99 & 329 & 36 & 0.86 & 0.95 & 5.4e2 & 0.96 & TO \\
				\bottomrule
		\end{tabular}}
		\caption{QPLIB instances solved by neither of the two methods.}
		\label{tableUnsolvedQPLIB}
	\end{table}

\newpage	\section{Comparison of the Cutting Plane Methods}
	\label{app:relaxboundDetails}
	
	Tables~\ref{tab:relaxboundDetail-boxqp} and~\ref{tab:relaxboundDetail-qplib} report detailed per-instance results for the four cutting-plane methods ``Dense McC.+Cuts'', ``Dense Cuts'', ``Sparse Cuts'', and ``SDP+Sparse Cuts'', previously summarized in aggregated form in Table~\ref{tab:bounds}.
	
	\begingroup
	\fontsize{7pt}{8pt}\selectfont
	\setlength{\tabcolsep}{1pt} 
	
	\begin{longtable}{l|rrrr|rrrr|rrr|rrrr}
		\caption{Comparison of the cutting plane methods on BoxQCQP instances.}\label{tab:relaxboundDetail-boxqp} \\
		\toprule
		& \multicolumn{4}{c|}{Dense McC.+Cuts} & \multicolumn{4}{c|}{Dense Cuts} & \multicolumn{3}{c|}{Sparse Cuts} & \multicolumn{4}{c}{SDP+Sparse Cuts} \\
		Instance & iter & cuts & GC & $t_{lastlp}$ & iter & cuts & GC & $t_{lastlp}$ & cuts & GC & $t_{lastlp}$ & $t_{SDP}$ & cuts & GC & $t_{lastlp}$ \\
		\midrule
		\endfirsthead
		\toprule
		& \multicolumn{4}{c|}{Dense McC.+Cuts} & \multicolumn{4}{c|}{Dense Cuts} & \multicolumn{3}{c|}{Sparse Cuts} & \multicolumn{4}{c}{SDP+Sparse Cuts} \\
		Instance & iter & cuts & GC & $t_{lastlp}$ & iter & cuts & GC & $t_{lastlp}$ & cuts & GC & $t_{lastlp}$ & $t_{SDP}$ & cuts & GC & $t_{lastlp}$ \\
		\midrule
		\endhead
		\midrule
		\multicolumn{16}{r}{Continued on next page} \\
		\midrule
		\endfoot
		\bottomrule
		\endlastfoot
		020-010-1\_10qc & 30 & 120 & 0.99 & 0 & 11 & 109 & 0.00 & 0 & 16 & 0.99 & 0.0 & 0.5 & 5 & 1.00 & 0.0 \\
		020-010-1\_5qc & 25 & 102 & 0.99 & 0 & 11 & 104 & 0.00 & 0 & 17 & 0.99 & 0.0 & 0.5 & 5 & 1.00 & 0.0 \\
		020-010-2\_10qc & 24 & 74 & 0.99 & 0 & 79 & 755 & 0.99 & 0 & 9 & 0.99 & 0.0 & 0.5 & 3 & 1.00 & 0.0 \\
		020-010-2\_5qc & 30 & 88 & 0.99 & 0 & 104 & 970 & 0.99 & 0 & 11 & 0.99 & 0.0 & 0.5 & 3 & 1.00 & 0.0 \\
		020-010-3\_10qc & 35 & 73 & 0.99 & 0 & 10 & 83 & 0.00 & 0 & 1 & 1.00 & 0.0 & 0.5 & 1 & 1.00 & 0.0 \\
		030-010-1\_10qc & 98 & 1107 & 0.99 & 1 & 354 & 4787 & 0.99 & 5 & 50 & 0.99 & 0.0 & 1.0 & 12 & 0.99 & 0.0 \\
		030-010-1\_5qc & 75 & 209 & 0.99 & 0 & 10 & 129 & 0.00 & 0 & 12 & 1.00 & 0.0 & 1.1 & 2 & 1.00 & 0.0 \\
		030-010-2\_10qc & 17 & 73 & 0.99 & 0 & 59 & 847 & 0.99 & 0 & 11 & 0.99 & 0.0 & 1.0 & 5 & 1.00 & 0.0 \\
		030-010-2\_5qc & 10 & 15 & 0.99 & 0 & 10 & 143 & 0.00 & 0 & 4 & 1.00 & 0.0 & 1.1 & 2 & 1.00 & 0.0 \\
		030-010-3\_10qc & 58 & 432 & 0.99 & 1 & 250 & 3448 & 0.99 & 4 & 46 & 0.99 & 0.0 & 1.0 & 12 & 1.00 & 0.0 \\
		040-010-1\_10qc & 75 & 1262 & 0.99 & 3 & 212 & 3911 & 0.99 & 11 & 26 & 0.99 & 0.0 & 2.0 & 11 & 0.99 & 0.0 \\
		040-010-1\_5qc & 144 & 2215 & 0.99 & 5 & 392 & 7040 & 0.98 & 20 & 34 & 0.99 & 0.0 & 2.0 & 10 & 1.00 & 0.0 \\
		040-010-2\_10qc & 112 & 1037 & 0.99 & 4 & 10 & 175 & 0.00 & 0 & 15 & 0.99 & 0.0 & 2.0 & 3 & 0.99 & 0.0 \\
		040-010-2\_5qc & 30 & 81 & 0.99 & 0 & 10 & 187 & 0.00 & 0 & 6 & 0.99 & 0.0 & 2.0 & 2 & 1.00 & 0.0 \\
		040-010-3\_10qc & 78 & 793 & 0.99 & 2 & 371 & 6822 & 0.99 & 21 & 35 & 0.99 & 0.0 & 2.0 & 12 & 0.99 & 0.0 \\
		040-010-3\_5qc & 42 & 476 & 0.99 & 2 & 222 & 4177 & 0.99 & 11 & 26 & 0.99 & 0.0 & 1.9 & 13 & 1.00 & 0.0 \\
		050-010-1\_10qc & 241 & 5006 & 0.98 & 26 & 259 & 6099 & 0.82 & 27 & 279 & 0.99 & 0.1 & 3.2 & 9 & 0.99 & 0.0 \\
		050-010-1\_5qc & 171 & 3568 & 0.99 & 18 & 256 & 6042 & 0.80 & 27 & 196 & 0.99 & 0.0 & 3.3 & 14 & 0.99 & 0.0 \\
		050-010-2\_10qc & 314 & 6578 & 0.97 & 23 & 249 & 5800 & 0.80 & 29 & 91 & 0.99 & 0.0 & 3.3 & 19 & 0.98 & 0.0 \\
		050-010-2\_5qc & 264 & 5785 & 0.98 & 28 & 255 & 5968 & 0.84 & 26 & 56 & 0.99 & 0.0 & 3.3 & 43 & 0.94 & 0.0 \\
		050-010-3\_10qc & 37 & 625 & 0.99 & 6 & 232 & 5457 & 0.98 & 30 & 19 & 0.99 & 0.0 & 3.5 & 18 & 0.87 & 0.0 \\
		050-010-3\_5qc & 19 & 330 & 0.99 & 3 & 11 & 258 & 0.00 & 1 & 11 & 0.99 & 0.0 & 3.5 & 20 & 0.99 & 0.0 \\
		060-010-1\_10qc & 184 & 5033 & 0.76 & 40 & 154 & 4378 & 0.44 & 36 & 467 & 0.99 & 0.1 & 6.0 & 7 & 0.99 & 0.0 \\
		060-010-1\_5qc & 182 & 4963 & 0.77 & 39 & 151 & 4308 & 0.45 & 32 & 461 & 0.99 & 0.1 & 5.7 & 8 & 0.99 & 0.0 \\
		060-010-2\_10qc & 181 & 4571 & 0.96 & 38 & 102 & 2925 & 0.33 & 28 & 132 & 0.99 & 0.1 & 5.9 & 32 & 0.96 & 0.0 \\
		060-010-2\_5qc & 23 & 257 & 0.99 & 2 & 92 & 2666 & 0.45 & 130 & 6 & 0.99 & 0.0 & 5.6 & 4 & 1.00 & 0.0 \\
		060-010-3\_10qc & 115 & 3022 & 0.59 & 71 & 144 & 4080 & 0.42 & 31 & 1071 & 0.99 & 0.3 & 5.8 & 5 & 0.99 & 0.0 \\
		060-010-3\_5qc & 122 & 3214 & 0.65 & 64 & 136 & 3848 & 0.42 & 28 & 620 & 0.99 & 0.2 & 5.9 & 15 & 0.96 & 0.0 \\
		070-010-1\_5qc & 134 & 4292 & 0.63 & 55 & 87 & 2907 & 0.20 & 114 & 1032 & 0.99 & 0.4 & 9.7 & 24 & 0.99 & 0.1 \\
		070-010-2\_10qc & 128 & 3994 & 0.69 & 48 & 90 & 3017 & 0.29 & 126 & 1035 & 0.99 & 0.4 & 9.5 & 8 & 0.99 & 0.1 \\
		070-010-3\_10qc & 45 & 1169 & 0.99 & 30 & 12 & 390 & 0.00 & 3 & 22 & 0.99 & 0.1 & 10.3 & 16 & 0.99 & 0.1 \\
		070-010-3\_5qc & 20 & 554 & 0.99 & 7 & 11 & 368 & 0.00 & 3 & 10 & 1.00 & 0.1 & 10.4 & 8 & 1.00 & 0.0 \\
		070-025-1\_10qc & 137 & 4407 & 0.94 & 52 & 122 & 4100 & 0.79 & 45 & 1728 & 0.96 & 2.2 & 11.3 & 6 & 0.99 & 0.2 \\
		070-025-2\_5qc & 141 & 4539 & 0.90 & 50 & 122 & 4063 & 0.75 & 43 & 1751 & 0.94 & 2.1 & 11.2 & 17 & 0.99 & 0.1 \\
		070-025-3\_10qc & 133 & 4272 & 0.97 & 53 & 122 & 4116 & 0.83 & 45 & 1709 & 0.97 & 2.5 & 11.5 & 10 & 0.99 & 0.2 \\
		070-025-3\_5qc & 139 & 4445 & 0.97 & 49 & 127 & 4277 & 0.85 & 43 & 1833 & 0.98 & 2.0 & 10.8 & 12 & 0.99 & 0.2 \\
		080-010-2\_10qc & 96 & 3490 & 0.74 & 63 & 78 & 3002 & 0.21 & 129 & 172 & 0.99 & 0.2 & 17.9 & 33 & 0.95 & 0.1 \\
		080-010-2\_5qc & 83 & 2990 & 0.70 & 90 & 76 & 2930 & 0.20 & 128 & 271 & 0.99 & 0.2 & 17.4 & 46 & 0.74 & 0.1 \\
		080-010-3\_10qc & 93 & 3387 & 0.83 & 76 & 72 & 2749 & 0.25 & 147 & 110 & 0.99 & 0.2 & 16.5 & 53 & 0.91 & 0.1 \\
		080-010-3\_5qc & 78 & 2771 & 0.74 & 90 & 72 & 2783 & 0.15 & 142 & 55 & 0.99 & 0.1 & 16.1 & 35 & 0.75 & 0.1 \\
		080-025-1\_10qc & 108 & 3972 & 0.90 & 60 & 75 & 2874 & 0.60 & 148 & 1128 & 0.92 & 2.0 & 16.9 & 9 & 0.99 & 0.3 \\
		080-025-1\_5qc & 96 & 3532 & 0.89 & 79 & 74 & 2821 & 0.62 & 135 & 1136 & 0.92 & 2.2 & 17.4 & 10 & 0.99 & 0.2 \\
		080-025-2\_10qc & 110 & 4040 & 0.83 & 64 & 76 & 2896 & 0.59 & 133 & 1177 & 0.87 & 2.0 & 17.1 & 12 & 0.99 & 0.2 \\
		080-025-2\_5qc & 95 & 3479 & 0.83 & 84 & 75 & 2861 & 0.57 & 155 & 1109 & 0.88 & 2.0 & 18.0 & 9 & 0.99 & 0.2 \\
		080-025-3\_10qc & 108 & 4002 & 0.90 & 65 & 77 & 2954 & 0.64 & 156 & 1147 & 0.93 & 2.1 & 17.7 & 13 & 0.99 & 0.3 \\
		080-025-3\_5qc & 108 & 3988 & 0.90 & 66 & 75 & 2889 & 0.63 & 152 & 1175 & 0.93 & 2.1 & 18.1 & 12 & 0.99 & 0.2 \\
		090-010-1\_10qc & 52 & 2145 & 0.46 & 240 & 64 & 2750 & 0.13 & 176 & 482 & 0.96 & 0.5 & 25.7 & 20 & 0.97 & 0.1 \\
		090-010-1\_5qc & 54 & 2198 & 0.49 & 220 & 60 & 2612 & 0.13 & 151 & 466 & 0.96 & 0.4 & 25.4 & 12 & 0.99 & 0.1 \\
		090-010-2\_10qc & 79 & 3251 & 0.58 & 92 & 72 & 3100 & 0.30 & 146 & 453 & 0.95 & 0.4 & 24.2 & 14 & 0.99 & 0.1 \\
		090-010-2\_5qc & 79 & 3244 & 0.60 & 92 & 73 & 3135 & 0.31 & 154 & 437 & 0.97 & 0.4 & 26.3 & 8 & 0.99 & 0.1 \\
		090-010-3\_10qc & 80 & 3311 & 0.54 & 91 & 71 & 3026 & 0.30 & 149 & 476 & 0.95 & 0.4 & 24.1 & 4 & 1.00 & 0.1 \\
		090-010-3\_5qc & 78 & 3231 & 0.56 & 98 & 71 & 3035 & 0.29 & 145 & 461 & 0.97 & 0.4 & 25.3 & 14 & 0.97 & 0.1 \\
		090-025-1\_10qc & 80 & 3306 & 0.86 & 86 & 62 & 2673 & 0.59 & 147 & 718 & 0.86 & 2.4 & 28.4 & 30 & 0.99 & 0.1 \\
		090-025-1\_5qc & 87 & 3582 & 0.88 & 78 & 65 & 2782 & 0.60 & 157 & 824 & 0.88 & 2.0 & 26.8 & 14 & 0.99 & 0.1 \\
		090-025-2\_10qc & 87 & 3606 & 0.87 & 87 & 68 & 2941 & 0.59 & 133 & 863 & 0.89 & 2.4 & 25.3 & 8 & 0.99 & 0.0 \\
		090-025-2\_5qc & 82 & 3414 & 0.87 & 89 & 61 & 2642 & 0.58 & 144 & 741 & 0.88 & 2.3 & 27.8 & 9 & 0.99 & 0.0 \\
		090-025-3\_10qc & 89 & 3702 & 0.91 & 84 & 70 & 3044 & 0.59 & 130 & 853 & 0.91 & 2.3 & 27.0 & 10 & 0.99 & 0.0 \\
		090-025-3\_5qc & 80 & 3328 & 0.91 & 93 & 62 & 2696 & 0.59 & 138 & 750 & 0.91 & 2.3 & 28.5 & 10 & 0.99 & 0.1 \\
		100-010-1\_10qc & 51 & 2332 & 0.55 & 338 & 54 & 2579 & 0.16 & 181 & 274 & 0.96 & 0.5 & 39.4 & 20 & 0.96 & 0.1 \\
		100-010-1\_5qc & 50 & 2284 & 0.56 & 332 & 52 & 2495 & 0.10 & 207 & 269 & 0.97 & 0.5 & 41.0 & 17 & 0.97 & 0.1 \\
		100-010-2\_10qc & 60 & 2763 & 0.59 & 164 & 63 & 2987 & 0.31 & 159 & 291 & 0.94 & 0.5 & 40.0 & 6 & 0.99 & 0.2 \\
		100-010-2\_5qc & 61 & 2804 & 0.66 & 148 & 63 & 2998 & 0.35 & 178 & 294 & 0.96 & 0.5 & 40.3 & 7 & 0.99 & 0.2 \\
		100-010-3\_10qc & 59 & 2755 & 0.67 & 156 & 59 & 2842 & 0.32 & 152 & 310 & 0.97 & 0.6 & 41.1 & 11 & 0.99 & 0.2 \\
		100-010-3\_5qc & 59 & 2741 & 0.67 & 154 & 55 & 2633 & 0.25 & 194 & 297 & 0.95 & 0.5 & 38.0 & 10 & 0.99 & 0.1 \\
		100-025-1\_10qc & 64 & 2938 & 0.84 & 140 & 53 & 2529 & 0.58 & 167 & 492 & 0.83 & 2.4 & 42.9 & 17 & 0.99 & 0.1 \\
		100-025-1\_5qc & 65 & 2987 & 0.84 & 151 & 53 & 2520 & 0.58 & 164 & 503 & 0.83 & 2.2 & 41.9 & 15 & 0.99 & 0.1 \\
		100-025-2\_10qc & 66 & 3047 & 0.81 & 146 & 53 & 2540 & 0.55 & 171 & 510 & 0.81 & 2.2 & 45.6 & 18 & 0.99 & 0.1 \\
		100-025-2\_5qc & 65 & 2980 & 0.84 & 142 & 52 & 2488 & 0.57 & 163 & 512 & 0.84 & 2.4 & 41.7 & 32 & 0.99 & 0.1 \\
		100-025-3\_10qc & 64 & 2957 & 0.85 & 139 & 50 & 2409 & 0.57 & 171 & 504 & 0.84 & 2.4 & 41.7 & 13 & 0.99 & 0.1 \\
		100-025-3\_5qc & 65 & 3005 & 0.86 & 137 & 52 & 2498 & 0.59 & 172 & 509 & 0.86 & 2.2 & 45.6 & 14 & 0.99 & 0.1 \\
		125-010-1\_10qc & 46 & 2645 & 0.50 & 127 & 43 & 2553 & 0.25 & 210 & 113 & 0.75 & 0.3 & 106.0 & 7 & 0.99 & 0.0 \\
		125-010-1\_5qc & 47 & 2702 & 0.49 & 127 & 44 & 2608 & 0.26 & 195 & 104 & 0.76 & 0.2 & 107.6 & 7 & 0.99 & 0.0 \\
		125-010-2\_10qc & 45 & 2610 & 0.62 & 133 & 42 & 2519 & 0.29 & 171 & 109 & 0.87 & 0.3 & 117.2 & 5 & 0.99 & 0.0 \\
		125-010-2\_5qc & 46 & 2675 & 0.63 & 132 & 38 & 2269 & 0.26 & 207 & 100 & 0.88 & 0.2 & 108.0 & 6 & 0.99 & 0.0 \\
		125-010-3\_10qc & 43 & 2493 & 0.61 & 138 & 46 & 2739 & 0.37 & 194 & 116 & 0.87 & 0.3 & 104.0 & 5 & 0.99 & 0.0 \\
		125-010-3\_5qc & 42 & 2440 & 0.63 & 150 & 44 & 2606 & 0.38 & 222 & 111 & 0.86 & 0.3 & 117.7 & 6 & 0.99 & 0.0 \\
		125-025-1\_10qc & 47 & 2708 & 0.77 & 120 & 34 & 2039 & 0.53 & 229 & 178 & 0.72 & 1.3 & 115.5 & 20 & 0.99 & 0.2 \\
		125-025-1\_5qc & 48 & 2774 & 0.77 & 126 & 35 & 2094 & 0.54 & 240 & 186 & 0.72 & 1.2 & 108.5 & 19 & 0.99 & 0.1 \\
		125-025-2\_10qc & 48 & 2755 & 0.84 & 135 & 33 & 1977 & 0.57 & 222 & 176 & 0.79 & 1.2 & 118.1 & 21 & 0.99 & 0.1 \\
		125-025-2\_5qc & 47 & 2705 & 0.84 & 124 & 33 & 1970 & 0.58 & 218 & 177 & 0.79 & 1.2 & 116.5 & 22 & 0.99 & 0.1 \\
		125-025-3\_10qc & 47 & 2705 & 0.85 & 131 & 35 & 2093 & 0.58 & 212 & 173 & 0.80 & 1.1 & 121.1 & 33 & 0.99 & 0.2 \\
		125-025-3\_5qc & 46 & 2649 & 0.86 & 141 & 34 & 2038 & 0.58 & 230 & 169 & 0.81 & 1.1 & 127.1 & 27 & 0.99 & 0.2 \\
		150-010-1\_10qc & 29 & 2024 & 0.48 & 235 & 33 & 2360 & 0.30 & 311 & 41 & 0.53 & 0.2 & 245.5 & 3 & 1.00 & 0.1 \\
		150-010-1\_5qc & 28 & 1941 & 0.48 & 232 & 33 & 2356 & 0.30 & 266 & 40 & 0.53 & 0.2 & 276.5 & 3 & 0.99 & 0.0 \\
		150-010-2\_10qc & 29 & 2019 & 0.51 & 236 & 33 & 2354 & 0.26 & 264 & 40 & 0.71 & 0.2 & 256.9 & 4 & 1.00 & 0.1 \\
		150-010-2\_5qc & 30 & 2083 & 0.52 & 223 & 33 & 2366 & 0.26 & 253 & 43 & 0.68 & 0.2 & 256.4 & 3 & 0.99 & 0.0 \\
		150-010-3\_10qc & 29 & 2023 & 0.53 & 228 & 32 & 2310 & 0.26 & 274 & 40 & 0.52 & 0.2 & 258.1 & 5 & 0.99 & 0.1 \\
		150-010-3\_5qc & 29 & 2032 & 0.56 & 220 & 31 & 2236 & 0.28 & 261 & 38 & 0.60 & 0.1 & 247.3 & 4 & 0.99 & 0.0 \\
		150-025-1\_10qc & 29 & 2016 & 0.84 & 224 & 22 & 1603 & 0.56 & 285 & 64 & 0.76 & 0.6 & 282.0 & 20 & 0.99 & 0.3 \\
		150-025-1\_5qc & 29 & 2018 & 0.85 & 220 & 21 & 1531 & 0.56 & 282 & 66 & 0.77 & 0.6 & 299.1 & 22 & 0.99 & 0.3 \\
		150-025-2\_10qc & 28 & 1959 & 0.84 & 231 & 23 & 1682 & 0.55 & 303 & 62 & 0.78 & 0.6 & 279.9 & 15 & 0.99 & 0.2 \\
		150-025-2\_5qc & 29 & 2025 & 0.85 & 250 & 23 & 1673 & 0.56 & 336 & 63 & 0.77 & 0.5 & 281.7 & 14 & 0.99 & 0.2 \\
		150-025-3\_10qc & 28 & 1950 & 0.84 & 231 & 20 & 1446 & 0.55 & 731 & 62 & 0.78 & 0.6 & 297.3 & 18 & 0.99 & 0.2 \\
		150-025-3\_5qc & 28 & 1951 & 0.84 & 242 & 21 & 1519 & 0.56 & 443 & 61 & 0.80 & 0.5 & 305.7 & 19 & 0.99 & 0.2 \\
		175-010-1\_10qc & 21 & 1716 & 0.48 & 385 & 26 & 2173 & 0.30 & 339 & 17 & 0.48 & 0.1 & 591.2 & 4 & 0.99 & 0.1 \\
		175-010-1\_5qc & 20 & 1637 & 0.48 & 419 & 25 & 2092 & 0.31 & 392 & 15 & 0.55 & 0.1 & 578.2 & 4 & 0.99 & 0.1 \\
		175-010-2\_10qc & 21 & 1721 & 0.40 & 354 & 28 & 2348 & 0.25 & 345 & 19 & 0.44 & 0.1 & 544.6 & 4 & 0.99 & 0.1 \\
		175-010-2\_5qc & 19 & 1562 & 0.40 & 443 & 25 & 2099 & 0.23 & 292 & 15 & 0.10 & 0.1 & 627.9 & 4 & 0.99 & 0.1 \\
		175-010-3\_10qc & 20 & 1645 & 0.39 & 431 & 26 & 2171 & 0.25 & 358 & 17 & 0.42 & 0.1 & 593.6 & 5 & 0.99 & 0.1 \\
		175-010-3\_5qc & 20 & 1636 & 0.40 & 413 & 25 & 2097 & 0.24 & 363 & 16 & 0.38 & 0.1 & 544.1 & 7 & 0.99 & 0.1 \\
		175-025-1\_10qc & 20 & 1635 & 0.80 & 446 & 15 & 1276 & 0.50 & 635 & 27 & 0.74 & 0.4 & 620.4 & 18 & 0.98 & 0.3 \\
		175-025-1\_5qc & 20 & 1627 & 0.81 & 424 & 16 & 1364 & 0.52 & 623 & 26 & 0.74 & 0.5 & 667.5 & 17 & 0.98 & 0.3 \\
		175-025-2\_10qc & 20 & 1639 & 0.83 & 410 & 17 & 1446 & 0.56 & 527 & 25 & 0.76 & 0.4 & 686.1 & 17 & 0.98 & 0.3 \\
		175-025-2\_5qc & 20 & 1636 & 0.83 & 416 & 15 & 1276 & 0.54 & 950 & 26 & 0.77 & 0.5 & 646.8 & 17 & 0.99 & 0.2 \\
		175-025-3\_10qc & 20 & 1633 & 0.80 & 399 & 16 & 1358 & 0.52 & 499 & 27 & 0.74 & 0.5 & 650.2 & 17 & 0.98 & 0.3 \\
		175-025-3\_5qc & 19 & 1550 & 0.80 & 472 & 16 & 1357 & 0.52 & 605 & 27 & 0.74 & 0.7 & 696.3 & 17 & 0.98 & 0.3 \\
		200-010-1\_10qc & 16 & 1510 & 0.40 & 752 & 19 & 1830 & 0.23 & 336 & 6 & 0.35 & 0.1 & 1300.8 & 3 & 0.99 & 0.1 \\
		200-010-1\_5qc & 16 & 1514 & 0.40 & 801 & 20 & 1930 & 0.25 & 343 & 7 & 0.35 & 0.1 & 1331.4 & 3 & 0.99 & 0.1 \\
		200-010-2\_10qc & 16 & 1517 & 0.41 & 730 & 18 & 1739 & 0.23 & 329 & 6 & 0.38 & 0.1 & 1229.2 & 3 & 0.97 & 0.1 \\
		200-010-2\_5qc & 15 & 1418 & 0.39 & 380 & 19 & 1834 & 0.25 & 354 & 6 & 0.37 & 0.2 & 1260.9 & 3 & 0.97 & 0.1 \\
		200-010-3\_10qc & 16 & 1509 & 0.39 & 841 & 20 & 1919 & 0.22 & 295 & 6 & 0.22 & 0.1 & 1241.9 & 3 & 0.99 & 0.1 \\
		200-010-3\_5qc & 16 & 1504 & 0.40 & 926 & 18 & 1731 & 0.19 & 343 & 6 & 0.07 & 0.1 & 1284.4 & 3 & 0.98 & 0.1 \\
		200-025-1\_10qc & 17 & 1606 & 0.80 & 740 & 13 & 1272 & 0.49 & 799 & 12 & 0.67 & 0.4 & 1306.9 & 7 & 0.93 & 0.2 \\
		200-025-1\_5qc & 17 & 1604 & 0.80 & 733 & 13 & 1271 & 0.50 & 1393 & 12 & 0.68 & 0.3 & 1362.6 & 6 & 0.93 & 0.2 \\
		200-025-2\_10qc & 17 & 1603 & 0.80 & 953 & 12 & 1172 & 0.48 & 413 & 12 & 0.68 & 0.4 & 1265.4 & 7 & 0.93 & 0.3 \\
		200-025-2\_5qc & 17 & 1607 & 0.80 & 743 & 12 & 1171 & 0.49 & 442 & 12 & 0.68 & 0.3 & 1325.2 & 6 & 0.92 & 0.2 \\
		200-025-3\_10qc & 17 & 1594 & 0.78 & 808 & 13 & 1270 & 0.48 & 1355 & 11 & 0.67 & 0.4 & 1367.6 & 5 & 0.91 & 0.3 \\
		200-025-3\_5qc & 17 & 1593 & 0.78 & 792 & 14 & 1363 & 0.50 & 766 & 12 & 0.68 & 0.3 & 1352.6 & 6 & 0.93 & 0.2\end{longtable}
	\endgroup
	
	\begingroup
	\fontsize{7pt}{8pt}\selectfont
	\setlength{\tabcolsep}{1pt} 

	
	\begin{longtable}{l|rrrr|rrrr|rrr|rrrr}
		\caption{Comparison of the cutting plane methods on QPLIB instances.}\label{tab:relaxboundDetail-qplib} \\
		\toprule
		& \multicolumn{4}{c|}{Dense McC.+Cuts} & \multicolumn{4}{c|}{Dense Cuts} & \multicolumn{3}{c|}{Sparse Cuts} & \multicolumn{4}{c}{SDP+Sparse Cuts} \\
		Instance & iter & cuts & GC & $t_{lastlp}$ & iter & cuts & GC & $t_{lastlp}$ & cuts & GC & $t_{lastlp}$ & $t_{SDP}$ & cuts & GC & $t_{lastlp}$ \\
		\midrule
		\endfirsthead
		\toprule
		& \multicolumn{4}{c|}{Dense McC.+Cuts} & \multicolumn{4}{c|}{Dense Cuts} & \multicolumn{3}{c|}{Sparse Cuts} & \multicolumn{4}{c}{SDP+Sparse Cuts} \\
		Instance & iter & cuts & GC & $t_{lastlp}$ & iter & cuts & GC & $t_{lastlp}$ & cuts & GC & $t_{lastlp}$ & $t_{SDP}$ & cuts & GC & $t_{lastlp}$ \\
		\midrule
		\endhead
		\midrule
		\multicolumn{16}{r}{Continued on next page} \\
		\midrule
		\endfoot
		\bottomrule
		\endlastfoot
		QPLIB\_0343 & 35 & 791 & 0.99 & 26 & 35 & 800 & 0.99 & 23 & 610 & 0.98 & 18.1 & 8.0 & 451 & 0.47 & 7.3 \\
		QPLIB\_0633 & 53 & 1394 & 0.69 & 156 & 53 & 1379 & 0.69 & 190 & 226 & 0.59 & 22.6 & 26.0 & 246 & 0.81 & 16.7 \\
		QPLIB\_0911 & 217 & 4854 & 0.99 & 17 & 215 & 4826 & 0.99 & 18 & 1449 & 0.92 & 4.8 & 8.6 & 1382 & 0.99 & 4.9 \\
		QPLIB\_0975 & 221 & 4909 & 0.99 & 12 & 215 & 4803 & 0.99 & 13 & 2000 & 0.94 & 3.6 & 6.7 & 1572 & 0.99 & 2.5 \\
		QPLIB\_10040 & 41 & 2437 & 0.82 & 189 & 44 & 2638 & 0.82 & 174 & 442 & 0.70 & 17.7 & 240.8 & 428 & 0.72 & 15.5 \\
		QPLIB\_10041 & 42 & 2421 & 0.74 & 137 & 43 & 2471 & 0.73 & 141 & 445 & 0.72 & 19.1 & 180.5 & 302 & 0.99 & 10.7 \\
		QPLIB\_10042 & 36 & 2013 & 0.94 & 256 & 41 & 2327 & 0.94 & 153 & 415 & 0.91 & 19.4 & 207.3 & 430 & 0.89 & 17.1 \\
		QPLIB\_10043 & 26 & 1893 & 0.80 & 254 & 27 & 1980 & 0.79 & 237 & 351 & 0.62 & 24.0 & 593.0 & 325 & 0.61 & 20.8 \\
		QPLIB\_10044 & 26 & 1639 & 0.79 & 290 & 24 & 1645 & 0.77 & 356 & 346 & 0.74 & 22.1 & 528.6 & 322 & 0.72 & 18.8 \\
		QPLIB\_10045 & 26 & 1844 & 0.64 & 281 & 29 & 2065 & 0.62 & 250 & 368 & 0.62 & 26.9 & 512.0 & 367 & 0.93 & 18.8 \\
		QPLIB\_10046 & 25 & 1617 & 0.86 & 354 & 25 & 1817 & 0.84 & 294 & 356 & 0.77 & 20.5 & 530.3 & 331 & 0.76 & 17.7 \\
		QPLIB\_10047 & 27 & 1905 & 0.79 & 243 & 27 & 1907 & 0.79 & 245 & 387 & 0.75 & 22.8 & 532.5 & 374 & 0.85 & 20.2 \\
		QPLIB\_10048 & 26 & 1784 & 0.85 & 280 & 27 & 1881 & 0.84 & 272 & 366 & 0.79 & 26.9 & 533.2 & 344 & 0.77 & 21.3 \\
		QPLIB\_10049 & 26 & 1828 & 0.79 & 258 & 27 & 1900 & 0.79 & 245 & 380 & 0.75 & 22.7 & 553.5 & 366 & 0.86 & 18.8 \\
		QPLIB\_10050 & 28 & 1846 & 0.89 & 231 & 29 & 1956 & 0.88 & 243 & 382 & 0.87 & 22.0 & 456.1 & 365 & 0.86 & 23.8 \\
		QPLIB\_10051 & 27 & 1879 & 0.78 & 248 & 28 & 1949 & 0.77 & 242 & 385 & 0.76 & 21.9 & 492.8 & 364 & 0.97 & 20.6 \\
		QPLIB\_10052 & 26 & 1718 & 0.85 & 263 & 25 & 1673 & 0.85 & 283 & 380 & 0.83 & 19.4 & 524.3 & 356 & 0.81 & 16.7 \\
		QPLIB\_10053 & 26 & 1855 & 0.90 & 243 & 27 & 1939 & 0.90 & 248 & 374 & 0.78 & 21.7 & 604.9 & 346 & 0.77 & 18.7 \\
		QPLIB\_10054 & 20 & 1484 & 0.77 & 609 & 18 & 1530 & 0.71 & 607 & 290 & 0.68 & 25.5 & 1157.9 & 237 & 0.63 & 18.0 \\
		QPLIB\_10055 & 17 & 1414 & 0.78 & 480 & 19 & 1624 & 0.78 & 423 & 304 & 0.49 & 22.7 & 1242.9 & 248 & 0.41 & 18.4 \\
		QPLIB\_10056 & 20 & 1573 & 0.81 & 423 & 19 & 1558 & 0.79 & 435 & 325 & 0.78 & 23.2 & 1023.0 & 279 & 0.76 & 20.5 \\
		QPLIB\_10057 & 20 & 1533 & 0.66 & 452 & 15 & 1479 & 0.51 & 440 & 161 & 0.29 & 14.1 & 2675.1 & 39 & 0.06 & 3.2 \\
		QPLIB\_10058 & 16 & 1503 & 0.71 & 867 & 15 & 1501 & 0.65 & 813 & 248 & 0.49 & 25.5 & 2989.2 & 83 & 0.18 & 6.4 \\
		QPLIB\_10059 & 15 & 1458 & 0.50 & 509 & 17 & 1687 & 0.39 & 570 & 151 & 0.61 & 8.0 & 1469.1 & 87 & 0.96 & 4.5 \\
		QPLIB\_10060 & 15 & 1484 & 0.60 & 697 & 14 & 1427 & 0.55 & 467 & 249 & 0.38 & 22.9 & 3038.0 & 78 & 0.12 & 6.6 \\
		QPLIB\_10061 & 16 & 1456 & 0.65 & 489 & 16 & 1561 & 0.58 & 744 & 266 & 0.63 & 28.4 & 2382.4 & 142 & 0.50 & 10.8 \\
		QPLIB\_10062 & 16 & 1544 & 0.51 & 665 & 15 & 1468 & 0.47 & 376 & 266 & 0.53 & 29.5 & 2111.3 & 171 & 0.85 & 13.2 \\
		QPLIB\_10063 & 17 & 1311 & 0.97 & 452 & 16 & 1284 & 0.96 & 511 & 213 & 0.95 & 41.2 & 2165.1 & 141 & 0.96 & 21.5 \\
		QPLIB\_10064 & 18 & 1468 & 0.73 & 463 & 17 & 1459 & 0.72 & 375 & 262 & 0.72 & 27.0 & 1947.5 & 169 & 0.94 & 16.5 \\
		QPLIB\_10065 & 16 & 1486 & 0.82 & 650 & 16 & 1506 & 0.81 & 688 & 274 & 0.78 & 26.1 & 2424.3 & 162 & 0.67 & 13.7 \\
		QPLIB\_10066 & 17 & 1549 & 0.86 & 729 & 16 & 1485 & 0.84 & 769 & 277 & 0.84 & 30.7 & 2428.3 & 165 & 0.76 & 12.6 \\
		QPLIB\_10067 & 16 & 1476 & 0.76 & 648 & 16 & 1524 & 0.74 & 674 & 278 & 0.73 & 26.0 & 2216.6 & 177 & 0.65 & 15.1 \\
		QPLIB\_10068 & 16 & 1498 & 0.81 & 728 & 16 & 1515 & 0.81 & 759 & 282 & 0.78 & 27.9 & 2396.9 & 166 & 0.68 & 13.4 \\
		QPLIB\_10069 & 14 & 1389 & 0.86 & 471 & 14 & 1399 & 0.85 & 477 & 278 & 0.46 & 28.4 & 3082.5 & 102 & 0.21 & 8.0 \\
		QPLIB\_10070 & 16 & 1508 & 0.74 & 642 & 16 & 1524 & 0.73 & 763 & 282 & 0.70 & 29.5 & 2343.1 & 170 & 0.60 & 14.3 \\
		QPLIB\_10071 & 17 & 1513 & 0.77 & 629 & 16 & 1494 & 0.77 & 591 & 286 & 0.68 & 26.9 & 2311.1 & 172 & 0.62 & 16.1 \\
		QPLIB\_10072 & 122 & 4036 & 0.79 & 58 & 118 & 4049 & 0.76 & 65 & 848 & 0.73 & 7.7 & 23.3 & 349 & 0.99 & 3.7 \\
		QPLIB\_10073 & 78 & 1812 & 0.82 & 185 & 91 & 2969 & 0.80 & 74 & 822 & 0.79 & 8.3 & 21.6 & 172 & 0.99 & 1.3 \\
		QPLIB\_10074 & 120 & 4130 & 0.93 & 61 & 119 & 4064 & 0.93 & 58 & 769 & 0.86 & 9.5 & 27.0 & 774 & 0.98 & 8.9 \\
		QPLIB\_1055 & 89 & 1556 & 0.99 & 2 & 88 & 1535 & 0.99 & 3 & 2127 & 0.99 & 3.5 & 4.8 & 557 & 0.99 & 0.8 \\
		QPLIB\_1143 & 129 & 2195 & 0.99 & 5 & 132 & 2276 & 0.99 & 4 & 766 & 0.94 & 10.6 & 4.8 & 687 & 0.98 & 9.5 \\
		QPLIB\_1157 & 27 & 364 & 0.99 & 4 & 27 & 385 & 0.99 & 5 & 255 & 0.99 & 2.8 & 4.5 & 95 & 0.99 & 1.2 \\
		QPLIB\_1353 & 40 & 836 & 0.99 & 30 & 42 & 888 & 0.99 & 29 & 577 & 0.98 & 17.4 & 8.1 & 292 & 0.96 & 5.0 \\
		QPLIB\_1423 & 68 & 1056 & 0.99 & 2 & 68 & 1048 & 0.99 & 2 & 740 & 0.98 & 11.2 & 4.8 & 227 & 0.99 & 2.2 \\
		QPLIB\_1437 & 181 & 3716 & 0.99 & 17 & 185 & 3822 & 0.99 & 21 & 595 & 0.90 & 16.4 & 7.8 & 598 & 0.99 & 17.0 \\
		QPLIB\_1451 & 71 & 1901 & 0.93 & 123 & 69 & 1853 & 0.93 & 121 & 463 & 0.80 & 12.6 & 13.0 & 449 & 0.85 & 11.5 \\
		QPLIB\_1493 & 98 & 1560 & 0.99 & 3 & 100 & 1580 & 0.99 & 3 & 798 & 0.96 & 9.9 & 4.3 & 302 & 0.99 & 2.3 \\
		QPLIB\_1507 & 58 & 800 & 0.99 & 1 & 58 & 791 & 0.99 & 0 & 983 & 0.99 & 0.6 & 2.5 & 315 & 0.99 & 1.2 \\
		QPLIB\_1535 & 64 & 1646 & 0.92 & 133 & 64 & 1644 & 0.92 & 118 & 370 & 0.81 & 14.4 & 15.4 & 376 & 0.85 & 15.4 \\
		QPLIB\_1619 & 101 & 1956 & 0.99 & 9 & 95 & 1857 & 0.99 & 8 & 556 & 0.94 & 16.1 & 8.9 & 568 & 0.98 & 17.1 \\
		QPLIB\_1661 & 71 & 1756 & 0.98 & 120 & 69 & 1715 & 0.98 & 103 & 453 & 0.89 & 12.6 & 12.5 & 456 & 0.98 & 11.0 \\
		QPLIB\_1675 & 69 & 1852 & 0.89 & 126 & 81 & 2229 & 0.87 & 108 & 602 & 0.73 & 18.4 & 10.9 & 602 & 0.99 & 17.3 \\
		QPLIB\_1703 & 65 & 1716 & 0.96 & 126 & 66 & 1759 & 0.96 & 128 & 402 & 0.89 & 12.6 & 14.8 & 416 & 0.94 & 13.1 \\
		QPLIB\_1745 & 165 & 3314 & 0.99 & 16 & 165 & 3315 & 0.99 & 15 & 542 & 0.88 & 16.4 & 8.9 & 542 & 0.96 & 17.0 \\
		QPLIB\_1773 & 68 & 1774 & 0.91 & 142 & 69 & 1787 & 0.91 & 125 & 462 & 0.82 & 11.6 & 12.9 & 466 & 0.85 & 11.5 \\
		QPLIB\_1886 & 166 & 3717 & 0.99 & 14 & 259 & 5913 & 0.99 & 24 & 2181 & 0.96 & 3.3 & 5.8 & 52 & 0.99 & 0.3 \\
		QPLIB\_1913 & 217 & 4674 & 0.99 & 17 & 304 & 6676 & 0.99 & 23 & 2466 & 0.90 & 3.0 & 5.1 & 93 & 0.99 & 0.2 \\
		QPLIB\_1922 & 53 & 694 & 0.99 & 0 & 101 & 1375 & 0.99 & 1 & 1051 & 0.99 & 0.3 & 1.8 & 18 & 0.99 & 0.0 \\
		QPLIB\_1931 & 91 & 1596 & 0.99 & 3 & 155 & 2810 & 0.99 & 5 & 2577 & 0.99 & 1.7 & 3.4 & 40 & 0.99 & 0.1 \\
		QPLIB\_1940 & 295 & 6307 & 0.99 & 18 & 321 & 7013 & 0.98 & 24 & 2428 & 0.86 & 3.3 & 5.2 & 178 & 0.99 & 0.3 \\
		QPLIB\_1967 & 97 & 2136 & 0.99 & 9 & 97 & 2137 & 0.99 & 8 & 1154 & 0.97 & 5.6 & 9.6 & 1089 & 0.99 & 5.1 \\
		QPLIB\_3523 & 20 & 1744 & 0.36 & 401 & 23 & 2001 & 0.14 & 328 & 12 & 0.35 & 0.2 & 846.8 & 9 & 0.50 & 0.1 \\
		QPLIB\_3714 & 19 & 885 & 0.17 & 270 & 40 & 1962 & 0.13 & 250 & 434 & 0.54 & 4.9 & 85.2 & 316 & 0.88 & 6.2 \\
		QPLIB\_3762 & 83 & 3580 & 0.72 & 80 & 61 & 2637 & 0.55 & 148 & 780 & 0.71 & 2.6 & 27.7 & 800 & 0.93 & 2.5 \\
		QPLIB\_3775 & 5 & 159 & 0.11 & 1898 & 18 & 1289 & 0.11 & 915 & 49 & 0.30 & 14.7 & 889.7 & 29 & 0.84 & 13.1 \\
		QPLIB\_3815 & 4 & 309 & 0.00 & 1289 & 10 & 875 & 0.00 & 25 & 7 & 0.51 & 0.1 & 879.5 & 2 & 0.99 & 0.1 \\
		QPLIB\_3834 & 273 & 4922 & 0.89 & 17 & 275 & 4944 & 0.88 & 17 & 635 & 0.75 & 14.8 & 8.0 & 640 & 0.97 & 14.2 \\
		QPLIB\_3870 & 20 & 1751 & 0.48 & 429 & 17 & 1495 & 0.26 & 494 & 19 & 0.39 & 0.3 & 817.2 & 13 & 0.48 & 0.2 \\
		QPLIB\_3883 & 20 & 1757 & 0.44 & 410 & 20 & 1768 & 0.23 & 395 & 15 & 0.34 & 0.2 & 831.0 & 11 & 0.51 & 0.1 \\
		QPLIB\_5875 & 17 & 1578 & 0.92 & 758 & 18 & 1689 & 0.88 & 513 & 223 & 0.85 & 15.3 & 1574.0 & 155 & 0.89 & 10.7 \\
		QPLIB\_5881 & 52 & 2890 & 0.86 & 126 & 39 & 2232 & 0.64 & 191 & 254 & 0.77 & 2.2 & 94.3 & 37 & 0.99 & 0.2 \\
		QPLIB\_5882 & 29 & 2001 & 0.90 & 260 & 35 & 2426 & 0.87 & 181 & 362 & 0.83 & 19.6 & 350.9 & 357 & 0.88 & 14.1 \\
		QPLIB\_5935 & 5 & 243 & 1.00 & 3 & 5 & 243 & 1.00 & 4 & 157 & 0.99 & 2.0 & 63.2 & 175 & 0.99 & 2.1 \\
		QPLIB\_5944 & 5 & 246 & 0.99 & 4 & 5 & 246 & 0.99 & 4 & 197 & 0.99 & 3.3 & 64.8 & 260 & 0.99 & 4.2 \\
		QPLIB\_5962 & 5 & 367 & 1.00 & 14 & 5 & 367 & 1.00 & 14 & 234 & 0.99 & 8.7 & 307.7 & 245 & 0.99 & 8.4 \\
		QPLIB\_5971 & 5 & 368 & 1.00 & 15 & 5 & 368 & 1.00 & 15 & 256 & 0.99 & 9.1 & 321.8 & 246 & 0.99 & 8.2 \\
		QPLIB\_5980 & 5 & 368 & 1.00 & 14 & 5 & 368 & 1.00 & 14 & 229 & 0.99 & 8.1 & 309.6 & 229 & 0.99 & 7.9\end{longtable}
	\endgroup

	\end{document}